\documentclass[preprint,12pt]{elsarticle}

\usepackage{array}
\usepackage{bigstrut}
\usepackage{graphicx}
\usepackage{epsfig}
\usepackage{amssymb}
\usepackage[intlimits]{amsmath}
\usepackage{amsmath}
\usepackage{slashbox}
\usepackage{multirow}
\usepackage{hhline}	
\usepackage{makecell}
\usepackage{color}
\usepackage{float}
\usepackage{graphicx,subfigure}
\usepackage{mathabx}

\newcommand{\q}{\quad}
\newcommand{\ee}{{\rm e}\hspace{1pt}}
\newcommand{\dd}{\hspace{0.5pt}{\rm d}\hspace{0.5pt}}
\newcommand{\real}{\hspace{0.5pt}{\rm Re}\hspace{0.5pt}}
\newcolumntype{R}[1]{>{\raggedleft\let\newline\\\arraybackslash\hspace{0pt}}m{#1}}
\newcolumntype{L}[1]{>{\raggedright\let\newline\\\arraybackslash\hspace{0pt}}m{#1}}
\newcolumntype{C}[1]{>{\centering\let\newline\\\arraybackslash\hspace{0pt}}m{#1}}

\definecolor{orange}{rgb}{1,0,0}

 \usepackage{amsthm}
\usepackage{lineno}
\usepackage{amssymb,latexsym,amscd,amsmath,amsfonts,enumerate,supertabular}
\usepackage{tabularx}

\journal{}

\numberwithin{equation}{section}
\newtheorem{lemma}[]{Lemma}
\numberwithin{lemma}{section}
\newtheorem{theorem}[]{Theorem}
\numberwithin{theorem}{section}

\numberwithin{example}{section}
\newcommand{\overbar}[1]{\mkern 1.5mu\overline{\mkern-1.5mu#1\mkern-1.5mu}\mkern 1.5mu}

\begin{document}
\begin{frontmatter}

\title{{\bf Preconditioned Implicit-Exponential (IMEXP) Time Integrators for Stiff  Differential Equations }}
\author{Vu Thai Luan \corref{cor1}}
\ead{vluan@ucmerced.edu}
\author{Mayya Tokman}
\ead{mtokman@ucmerced.edu}
\author{Greg Rainwater }
\ead{grainwater@ucmerced.edu} 
\address{School of Natural Sciences, University of California, Merced, 5200 North Lake Road,\\
Merced, CA 95343, USA}
\cortext[cor1]{Corresponding author}

\begin{abstract}
\small  
\indent  We propose two new classes of time integrators for stiff DEs: the implicit-explicit exponential (IMEXP) and the hybrid exponential methods. In contrast to the existing exponential schemes, the new methods offer significant computational advantages when used with preconditioners. Any preconditioner can be used with any of these new schemes. This leads to a broader applicability of exponential methods. The proof of stability and convergence of these integrators and numerical demonstration of their efficiency are presented. 
 
\small  
\end{abstract}

\begin{keyword}
\small 
 Exponential integrators  \sep Stiff differential equations \sep Implicit-explicit exponential \sep Preconditioner 
\end{keyword}

\end{frontmatter}
\section{Introduction}
Many problems in science and engineering are characterized by the presence of a wide range of spatial and temporal scales in the phenomenon under investigations.  Complex interactions of numerous processes evolving on different scales can cause the differential equations describing the evolution of the system to be stiff.   Solving such stiff large systems of differential equations numerically is a challenging task.  In particular, in many applications very large systems of partial or ordinary differential equations have to be solved numerically over very long time intervals compared to the fastest processes in the system.  Development of an efficient time integrator that enables simulation of such system in a reasonable time requires much effort and care since standard methods can be too computationally expensive.  A custom designed time integrator which exploits the structure of the problem and the source of stiffness can bring the necessary computational savings that enable simulation of the problem in the parameter regimes of interest.  In this paper we address a class of initial value problems which can be written in the form 
\begin{equation}
u'(t)=F(u(t))= L(u(t)) + N(u(t)), \qquad u(t_0)=u_0,  \label{eq1.1}
\end{equation}
where both differential operators $L$ and $N$ can be stiff. Often $L(u(t)) = Lu(t)$ is a linear operator that represents, for instance diffusion.  If $N(u(t))$ is not a stiff operator, equations of type \eqref{eq1.1} are usually solved using implicit-explicit (IMEX) integrators.  IMEX schemes have been widely used in a variety of fields with some of the earlier applications coming from fluid dynamics in conjunction with spectral methods \cite{canuto87, kim85}. An example of one of the most widely used, particularly in the context of large-scale applications, IMEX schemes, is the second order BDF-type method (we will call it here 2-sBDF) which was proposed in \cite{AscherRuuth}, one of the first publications where IMEX methods were systematically analyzed and derived.  Over the past decades a range of IMEX schemes have been introduced such as, for example, linear multistep \cite{AscherRuuth, frank1997}, Runge-Kutta \cite{ascher1997} and extrapolated \cite{constantinescu2010} IMEX methods.  Such schemes have proven to be very efficient for problems such as advection-diffusion equations or reaction-diffusion equations where advection or reaction are slow and diffusion is occurring on a fast time scale.  Diffusion in this case is treated implicitly while advection or reaction terms are treated with an explicit method. The IMEX methods are particularly efficient if a good preconditioner is available to speed up convergence of an iterative method used in the implicit solver.  Construction of an efficient preconditioner is the topic of extensive research and software development; frequently the majority of time spent on development and implementation of an IMEX scheme for a large scale problem goes to creating a preconditioner \cite{wathen15}.  The complexity of a differential operator that has to be preconditioned is directly related to the difficulty in constructing an efficient preconditioner.  For example, a large number of preconditioners have been developed for a Laplacian operator which models linear diffusion process.  

While IMEX schemes work well if $L$ is a stiff operator and $N$ is not, in many applications both of these terms introduce stiffness.  Such problems arise from a wide range of fields, from electrochemistry \cite{SBL12} to combustion \cite{bisetti} and plasma physics \cite{tokman2001}.  A reaction-diffusion system describing 
chemical kinetic mechanisms involved in ignition of different chemical mixtures can involve thousands of reactions occurring over a wide range of time scales comparable to those of the diffusive processes in the system \cite{bisetti}. Similar structure can be encountered in models of electrochemical material growth where for certain parameter regimes the reactive terms can be as stiff as the diffusive operators in the equations \cite{SBL12}.  In magnetohydrodynamic (MHD) equations describing the large scale plasma behavior, stiffness arises from the presence of a wide range of waves encoded in the complex nonlinear terms of the system \cite{goedbloed2004}.  While IMEX or closely-related semi-implicit integrators are typically used for these problems, their performance suffers. The stiffness of the nonlinear operator $N(u(t))$ which is treated explicitly imposes prohibitive stability restrictions on the time step size.  Abandoning IMEX approach in this case and using a method that treats $N(u(t))$ implicitly as well, also poses a computational challenge.  Operator $N(u(t))$ can be very complex and development of an efficient preconditioner to enable implicit treatment of this term might be difficult or even impossible.  

Recently exponential integrators emerged as an efficient alternative to standard time integrators for solving stiff systems of equations. It has been shown that exponential time integrators can offer significant computational savings particularly for large scale stiff systems, particularly in comparison to implicit Newton-Krylov methods \cite{hochbruckexp4, Tok06,  HOS09, HO10, TLP12, LO14a, LO16, loffeld14}.  However, such comparisons are valid for problems where no efficient preconditioner is available for the implicit Newton-Krylov integrators.  A shortcoming of the exponential integrators is that, at present, there are no algorithms that can utilize preconditioners in a way that makes them clearly competitive with the preconditioned implicit Newton-Krylov methods.  In this paper we present a new class of implicit-exponential (IMEXP) methods which can both -- take advantage of efficient preconditioners developed for given operators and improve computational efficiency for problems where both operators $L$ and $N$ in \eqref{eq1.1} are stiff.  The idea of combining an implicit treatment of operator $L$ and an exponential approach to integrating term $N$ was first proposed in \cite{tokmanOberwolfach} where a classically accurate second order IMEXP method was constructed.  Here we expand this approach to derive several types of IMEXP integrators and provide derivation of stiffly accurate schemes along with the convergence theory for these methods.  While we propose two main classes of integrators -- IMEXP Runge-Kutta and Hybrid IMEXP schemes -- the ideas behind these methods can be used to construct many additional schemes that would address particular structure of the problem \eqref{eq1.1}. 

The paper is organized as follows.  Section 2 outlines the main ideas behind construction of IMEXP schemes and presents the analytical framework that enables us to derive the stiff order conditions and to prove stability and convergence of the schemes.  Construction and analysis of IMEXP Runge-Kutta methods is presented in Section 3 and development of hybrid IMEXP schemes is the focus of Section 4.  Section 5 contains numerical experiments that validate theoretical results and illustrate computational savings that IMEXP methods can bring compared to IMEX schemes for problems with stiff $N$.

\section{Construction and analytical framework for analysis of the IMEXP methods.}
We begin construction of IMEXP methods by considering the general EPIRK framework introduced in \cite{Tok11}.  The exponential propagation iterative methods of Runge-Kutta type (EPIRK) to solve \eqref{eq1.1} can be written as 
\begin{equation}
\label{eqn:genEPIRK}
  \begin{aligned}
U_{ni}&= u_n+\alpha_{i1}\psi_{i1}(g_{i1}h_nA_{i1})h_nF(u_n)+h_n\sum_{j=2}^{i-1} \alpha_{ij}\psi_{ij}(g_{ij}h_nA_{ij})D_{nj},\quad i=2,\dots,s,\\
u_{n+1} &= u_n +\beta_{1}\psi_{s+1,1}(g_{s+1,1}h_nA_{s+1,1})h_nF(u_n)+h_n\sum_{j=2}^s \beta_j\psi_{s+1j}(g_{s+1j}h_nA_{s+1,j})D_{n,s+1},
  \end{aligned}
\end{equation}
where $u_n$ is an approximation to the solution of \eqref{eq1.1} at time $t_n = t_0+\sum_{i=1}^n h_i$. 
As explained in \cite{Tok11, tokmanOberwolfach, rainwater2016}, different choices for functions $\psi_{ij}$, matrices $A_{ij}$ and vectors $D_{nj}$ result in different classes of EPIRK methods. 
To construct IMEXP methods we can use the flexibility of EPIRK framework and choose $\psi_{ij}$, $A_{ij}$ and $D_{nj}$ to address the structure of the problem \eqref{eq1.1}.  Namely, we construct methods of two types -- IMEXP Runge-Kutta and hybrid IMEXP schemes.  The main idea behind both classes of methods is to choose some of the functions $\psi_{ij}$ to be rational functions, similar to implicit or IMEX methods. The remaining $\psi_{ij}$ are then set as linear combination of the exponential functions $\varphi_k(z)$ defined by 
\begin{equation} \label{eq2.2}
\varphi_{0}(z)=\ee^z, \q
\varphi_{k}(z)=\int_{0}^{1} \ee^{(1-\theta )z} \frac{\theta ^{k-1}}{(k-1)!}\,\text{d}\theta , \quad k\geq 1.
\end{equation}
The arguments $A_{ij}$ of these functions can be chosen as either the full Jacobian $J_n=DF(u_n)=F'(u_n)=L'(u_n) + N'(u_n)=L'_n + N'_n$  or its components $L'$ or $N'$.  As noted in \cite{Tok11, rainwater14, rainwater2016} vectors $D_{nj}$ could be node values, differences or forward differences $\Delta^{j}$  constructed using the remainder functions such as   $F(u) - F(u_n)- F'(u_n)(u-u_n)$ or its component $N(u)$.  In order to simplify analysis we will restrict our choice to $D_{nj}$ being the difference constructed using the nonlinearity $N(u)$ as explained below.  In the next section we present three classes of methods that are defined by these choices of functions $\psi_{ij}$ and their arguments $A_{ij}$.  

\subsection{Construction of IMEX Runge-Kutta methods}

The IMEXP Runge-Kutta methods are designed for problems of type
\begin{equation} \label{eqIMEXPRK}
u'(t)=F(u(t))= Lu(t) + N(u(t)), \qquad u(t_0)=u_0,
\end{equation}
where $L$ is a stiff linear operator and the nonlinear operator $N(u(t))$ is either nonstiff or mildly stiff compared to $L$.  In this case we choose 
\begin{align}
&\psi_{i1}(z)=R_{i1}(z)=\frac{\sum_{k=0}^{m_i} p_{ik}z^{k}}{\sum_{k=0}^{n_i} q_{ik} z^k}  , \quad i = 1,...,s \\
&\psi_{s+1,1}(z) = R_1(z)=\frac{\sum_{k=0}^{m}  p_{k}z^{k}}{\sum_{k=0}^{n} q_{k} z^k},   \label{eqRni}
\end{align}
i.e. where $R_{i1}(z)$ and $R_1(z)$ are rational functions.  The remaining functions $\psi_{ij}(z)$, $j \ge 2$ are chosen to be linear combinations of $\varphi_k(z)$ functions defined in \eqref{eq2.2}.  To simplify the notation, and in anticipation of our analysis approach described below, we denote 
\begin{equation}
\begin{aligned} \label{eqab}
&a_{ij}(z) = \alpha_{ij} \psi_{ij}(z), \quad j = 2,...,s, \quad i = 1,...,s, \\
&b_j(z) = \beta_j \psi_{s+1,j}(z), \quad j = 2,...,s. 
\end{aligned}
\end{equation}
Since $L$ is the main source of stiffness in this type of problems we use it as the argument of all functions $\psi_{ij}(z)$.  The remainder, or defect, vectors $D_{ni}$ are chosen as 
\begin{equation}
D_{ni} = N(U_{ni}) - N(u_n).
\end{equation}
Here we consider the constant time step version of the method and set $h_n=h$.  The resulting general form of the IMEXP Runge-Kutta schemes is then given by 
\begin{subequations} \label{eq2.5}
\begin{align}
 U_{ni}&= u_n + c_i h R_{1i} ( c_i h L)F(u_n) +
 h \sum_{j=2}^{i-1}a_{ij}(h L) D_{nj}, \  2\leq i\leq s,  \label{eq2.5a} \\
u_{n+1}& = u_n + h R_{1} ( h L)F(u_n) + h \sum_{i=2}^{s}b_{i}(h L) D_{ni},  \label{eq2.5b}
\end{align}
\end{subequations}
where we have also used the simplifying assumptions 
\begin{subequations} \label{agc}
\begin{align}
&\alpha_{i1} = g_{i1} = c_i \label{agc_a} \\
&\alpha_{s+1,1}=g_{i1} = 1 \label{agc_b}.
\end{align}
\end{subequations}
\eqref{agc_b} is motivated by Lemma 4 in \cite{rainwater2016} which shows that this requirement is a necessary condition to satisfy the stiff order conditions.  As shown in \cite{rainwater2016}, assumption \eqref{agc_a} can potentially be relaxed to derive more methods, but we choose it for simplicity of subsequent analysis.  

\subsection{Construction of hybrid IMEXP methods}
We construct the hybrid IMEXP methods to address problems \eqref{eqIMEXPRK} where both operator $L$ and the nonlinearity $N(u(t))$ are stiff.   Since the operator $N(u)$ is stiff, it is desirable to treat $N'(u)$ in an implicit or exponential way.  We assume that implicit treatment is difficult in this case due to the lack of a readily available efficient preconditioner  which is crucial to making the implicit method sufficiently fast. Therefore we will use the exponential-type approach in constructing the integrator. We note that for stability reasons the product of the Lipschitz constant of  $N(u(t))$ and the time step must be sufficiently small. To improve the stability of an integrator, one can use the idea of linearizing $F(u)$ continuously along the numerical solution, which can make coefficients of leading error terms become smaller in each integration step. This motivates two possibilities-- either we use the full Jacobian 
\begin{equation}
J_n = J(u_n)=F'(u_n) = L + N'(u_n)=L + N'_n
\end{equation}
or just its nonlinear part $N'(u_n)$ as an argument for functions $a_{ij}(z)$ and $b_j(z)$ in \eqref{eqab}.  As before we assume that a rational function of the stiff operator $L$ can be computed efficiently, e.g. due to the availability of an efficient preconditioner.  Thus we get a class of hybrid IMEXP methods that can be written as 
\begin{subequations} \label{eq2.7}
\begin{align}
 U_{ni}&= u_n + c_i h R_{1i} ( c_i h L)F(u_n) +
 h \sum_{j=2}^{i-1}a_{ij}(h J_n) D_{nj}, \  2\leq i\leq s,  \label{eq2.7a} \\
u_{n+1}& = u_n + h R_{1} ( h L)F(u_n) + h \sum_{i=2}^{s}b_{i}(h J_n) D_{ni}  \label{eq2.7b}
\end{align}
with
\begin{equation} \label{eq2.7c}
  D_{ni}= N( U_{ni})- N(u_n ), \qquad  2\leq i\leq s.
\end{equation}
\end{subequations}
If the nonlinear portion of the Jacobian $N'(u)$ is used we get a modified hybrid IMEXP method
\begin{subequations} \label{eq2.8}
\begin{align}
 U_{ni}&= u_n + c_i h R_{1i} ( c_i h L)F(u_n) +
 h \sum_{j=2}^{i-1}a_{ij}(h N'_n) D_{nj}, \  2\leq i\leq s,  \label{eq2.8a} \\
u_{n+1}& = u_n + h R_{1} ( h L)F(u_n) + h \sum_{i=2}^{s}b_{i}(h N'_n) D_{ni}  \label{eq2.8b}
\end{align}
with
\begin{equation} \label{eq2.8c}
  D_{ni}= N( U_{ni})- N(u_n ), \qquad  2\leq i\leq s.
\end{equation}
\end{subequations}
It is possible that scheme \eqref{eq2.8} can be beneficial in cases where evaluation of $N'(u)v$ for a vector $v$ is significantly less computationally expensive than computing a product of a full Jacobian with a vector $J(u)v$. However, as we will see from the numerical experiments this has to be counterbalanced by some loss of accuracy compared to scheme \eqref{eq2.7}

\subsection{Analytical framework for derivation of the stiff order conditions and convergence analysis}
A closely related class of methods to the EPIRK schemes is the exponential Runge--Kutta integrators \cite{HO05, LO12b, LO14b}.  Exponential Runge--Kutta methods can be viewed as a special case of EPIRK schemes with simplifying assumptions on the coefficients (see \cite{rainwater14}). A theory of stiff order conditions and convergence analysis for these methods were developed in a series of papers \cite{HO05, LO12b, LO14b}.  The key idea in our derivation of the stiff order conditions and proof of convergence for the IMEXP schemes is their interpretation as a perturbation of the exponential Runge--Kutta methods.  This approach allows us to derive estimates for the error of the IMEXP methods, prove their stability and convergence.  

As for exponential Runge--Kutta methods, the convergence analysis of IMEXP schemes \eqref{eq2.5}, \eqref{eq2.7} and \eqref{eq2.8} can be carried out in the abstract framework of strongly continuous semigroups in a Banach space $X$ with norm $\| \cdot \|$ (for instance, see \cite{EN2000,PAZY83}). As usual in exponential integrators, this framework allows us to handle stiff problems. In particular, throughout the paper we will make use of the following main assumption.

{\em Assumption 1. The linear operator $L$ is the generator of a strongly continuous semigroup  $\ee^{tL}$ in $X$}.

This assumption implies that there exist constants $C$ and $\omega$ such that
\begin{equation} \label{eq3.2}
\|\ee^{tL}\|_{X\leftarrow X}\leq C\ee^{\omega t}, \quad t\geq 0
\end{equation}
holds uniformly in a neighborhood of the exact solution, leading to the boundedness of coefficients $a_{ij}(hL)$ and $b_{i}(hL)$.


In the subsequent analysis we will investigate questions on the consistency and stability of the proposed IMEXP methods and derive specific schemes.

\section{IMEXP Runge--Kutta methods }\label{sc4}
Through out this section the nonlinearity $N(u)$ is supposed to be a nonstiff with a moderate Lipschitz constant. We thus can make use of the following additional regularity assumption (for instance, see \cite{LO14b}) in order to study the local error of the proposed ansatz \eqref{eq2.5}.

{\em Assumption 2. We suppose that \eqref{eq1.1} possesses a sufficiently smooth solution $u: [0, T]\rightarrow X$, with derivatives in $X$ and that  $N: X \rightarrow X$ is sufficiently often Fr\'echet differentiable in a strip along the exact solution. All occurring derivatives are assumed to be uniformly bounded.} 

Assumption 2 implies that $N$ is locally Lipschitz in a strip along the exact solution. Typical examples are semilinear reaction-diffusion-advection equations, see \cite{H81}.

 \subsection{Local error analysis}
 Let $\tilde{u}_n$ denote the exact solution of \eqref{eq1.1} at time $t_n$, i.e., $\tilde{u}_n=u(t_n)$. Let 
\begin{equation} \label{eq3.1}
\overbar{e}_{n+1}=\overbar{u}_{n+1}- \tilde{u}_{n+1}
\end{equation}
denote the local error of \eqref{eq2.5} at $t_{n+1}$.
To analyze $\overbar{e}_{n+1}$, we consider one step integration starting with the initial value $\tilde{u}_n$ being the exact solution.

\begin{subequations} \label{eq3.3}
\begin{align}
\overbar{U}_{ni}&=\tilde{u}_n + c_i h R_{1i}(c_i h L)F(\tilde{u}_n) + h \sum_{j=2}^{i-1}a_{ij}(h L)  \overbar{D}_{nj}, \label{eq3.3a}\\
\overbar{u}_{n+1}&= \tilde{u}_n + h R_{1} (h L)F(\tilde{u}_n) + h \sum_{i=2}^{s}b_{i}(h L)  \overbar{D}_{ni} \label{eq3.3b}
\end{align}

with
\begin{equation} \label{eq3.3c}
\overbar{D}_{ni}= N( \overbar{U}_{ni})- N(\tilde{u}_n ).
\end{equation}
\end{subequations}
Taking a closer look at \eqref{eq3.3}, one can consider it as a perturbation scheme of the exponential Runge--Kutta scheme 
\begin{subequations} \label{eq3.4}
\begin{align}
\widecheck{U}_{ni}&=\tilde{u}_n + c_i h\varphi_1(c_i h L)F(\tilde{u}_n) + h \sum_{j=2}^{i-1}a_{ij}(h L)  \widecheck{D}_{nj}, \label{eq3.4a}\\
\check{u}_{n+1}&= \tilde{u}_n + h \varphi_1 (h L)F(\tilde{u}_n) + h \sum_{i=2}^{s}b_{i}(h L)  \widecheck{D}_{ni} \label{eq3.4b}
\end{align}
with
\begin{equation} \label{eq3.4c}
\widecheck{D}_{ni}= N( \widecheck{U}_{ni})- N(\tilde{u}_n ).
\end{equation}
\end{subequations}
This suggests us to employ the result on local errors of exponential Runge--Kutta methods (see \cite{LO12b, LO13}) for studying the local errors $\overbar{e}_{n+1}$ of scheme \eqref{eq2.5}. First, we have the following observation.

Denoting $\widecheck{N}'_{ni}=N'( \widecheck{U}_{ni} )$, $ \widecheck{E}_{ni}= \overbar{U}_{ni}-\widecheck{U}_{ni}$  and using the Taylor series expansion of $N(u)$, we get
\begin{equation} \label{eq3.5}
\overbar{D}_{ni}- \widecheck{D}_{ni}=N(\overbar{U}_{ni})-N( \widecheck{U}_{ni} )=\widecheck{N}'_{ni} \widecheck{E}_{ni} + \widecheck{R}_{ni} 
\end{equation}
with remainder
\begin{equation} \label{eq3.6}
\widecheck{R}_{ni}=\int_{0}^{1} (1-\theta ) N''(\widecheck{U}_{ni} + \theta \widecheck{E}_{ni})(\widecheck{E}_{ni}, \widecheck{E}_{ni})
\text{d}\theta. 
\end{equation}
Employing Assumption~2 shows that
\begin{equation} \label{eq3.7}
\| \widecheck{R}_{ni}\|\leq C\|\widecheck{E}_{ni}\|^2,
\end{equation}
as long as $\overbar{E}_{ni}$ remain in a sufficiently small neighborhood of 0. \\
By subtracting \eqref{eq3.4a} from \eqref{eq3.3a} and denoting  $\widecheck{K}_{ni}=\overbar{D}_{ni}- \widecheck{D}_{ni}$, we get
\begin{equation} \label{eq3.8}
\widecheck{E}_{ni}=c_i h  (R_{1i} (c_i h L)-\varphi_1 (c_i h L))F(\tilde{u}_n)+h\sum_{j=2}^{i-1}a_{ij}(h L) \widecheck{K}_{nj}.
\end{equation}
Inserting \eqref{eq3.8} into \eqref{eq3.5} gives
\begin{equation} \label{eq3.9}
\widecheck{K}_{ni}=c_i h \widecheck{N}'_{ni} (R_{1i} (c_i h L)-\varphi_1 (c_i h L))F(\tilde{u}_n)+h\sum_{j=2}^{i-1}a_{ij}(h L) \widecheck{K}_{nj}+\widecheck{R}_{ni}.
\end{equation}
With this preparation at hand, we are ready to state the following result concerning the stiff order conditions for IMEXP Runge--Kutta methods
\begin{theorem}\label{th3.1}
Under Assumptions 1 and 2, an implicit-explicit exponential Runge--Kutta method has order of consistency $p+1$ if its coefficients $a_{ij}(hL), b_i (hL)$ satisfy the stiff order conditions of the exponential Runge--Kutta methods up to some order $p$ and if coefficients $R_1(h L)$ and $R_{1i}(c_i h L)$ stated in \eqref{eqRni} are chosen in such L way that
\begin{subequations}  \label{eq3.10}
\begin{align}
(R_{1i} (c_i h L)-\varphi_1 (c_i h L))F(\tilde{u}_n)&=\mathcal{O}(h^{p-1}), \label{eq3.10a} \\
 (R_1(h L)-\varphi_1 (h L))F(\tilde{u}_n)&=\mathcal{O}(h^{p})\label{eq3.10b}.
 \end{align}
\end{subequations}
\end{theorem}
\begin{proof}
First, we note that the difference $\check{u}_{n+1}-\tilde{u}_{n+1}$ is the local errors of the exponential Runge--Kutta methods. 
It is proved in \cite{LO13} that if coefficients $a_{ij}(hL), b_i (hL)$ satisfy the stiff order conditions of the exponential Runge--Kutta methods of order $p$ then 
 \begin{equation}  \label{eq3.11}
\check{u}_{n+1}-\tilde{u}_{n+1}=\mathcal{O}(h^{p+1}).
\end{equation}
We then express the local errors $\overbar{e}_{n+1}$ given by \eqref{eq3.1} (applied to scheme  \eqref{eq2.5}) as
 \begin{equation}  \label{eq3.12}
\overbar{e}_{n+1}=(\overbar{u}_{n+1}-\check{u}_{n+1})+ (\check{u}_{n+1}-\tilde{u}_{n+1})=\overbar{u}_{n+1}-\check{u}_{n+1}+ \mathcal{O}(h^{p+1}).
\end{equation}
Subtracting \eqref{eq3.4b} from \eqref{eq3.3b} and using \eqref{eq3.5}, \eqref{eq3.9} gives us
 \begin{equation}  \label{eq3.13}
 \begin{aligned}
\overbar{u}_{n+1}-\check{u}_{n+1}&=h\big(R_1(h L)-\varphi_1 (h L)\big)F(\tilde{u}_n) \\
&+c_i h^2 \sum_{i=2}^{s}b_{i}(h L) \widecheck{N}'_{ni} \big(R_{1i} (c_i h L)-\varphi_1 (c_i h L)\big)F(\tilde{u}_n)\\
&+ h^2 \sum_{i=2}^{s}b_{i}(h L) \big(\sum_{j=2}^{i-1}a_{ij}(h L) \widecheck{K}_{nj}+\widecheck{R}_{ni} \big).
\end{aligned}
\end{equation}
Under assumptions in \eqref{eq3.10} and using \eqref{eq3.6}--\eqref{eq3.9}, it is straightforward to derive
$ \widecheck{K}_{ni}=\mathcal{O}(h^{p}), \  \widecheck{R}_{ni}=\mathcal{O}(h^{2p}) $ and thus we deduce from \eqref{eq3.13} that
  \begin{equation}  \label{eq3.14}
\overbar{u}_{n+1}-\check{u}_{n+1}=\mathcal{O}(h^{p+1}).
\end{equation}
From \eqref{eq3.12} and \eqref{eq3.14}, it shows at once $\overbar{e}_{n+1}=\mathcal{O}(h^{p+1}).$

\end{proof}
\subsection{Stability and convergence results}
Let $ e_{n+1} = u_{n+1} - u(t_{n+1})=u_{n+1} - \tilde{u}_{n+1}$ denote the global error of scheme \eqref{eq2.5} at time $t_{n+1}$, and let $\overbar{E}_{ni}=U_{ni}-\overbar{U}_{ni}$. 
It is easy to see that
\begin{equation} \label{eq3.15}
e_{n+1}=u_{n+1}-\bar{u}_{n+1}+\bar{u}_{n+1}- \tilde{u}_{n+1}=u_{n+1}-\bar{u}_{n+1}+\overbar{e}_{n+1}.
\end{equation}
Subtracting \eqref{eq3.3b} from \eqref{eq2.5b} and using $F(u)=Lu+N(u)$ shows that
\begin{equation} \label{eq3.16}
u_{n+1}-\bar{u}_{n+1}=\big(I+R_1(hL)hL\big)e_n + h\big( N(u_n)-N(\tilde{u}_{n})+ \sum_{i=2}^{s}b_{i}(h L) \overbar{K}_{ni} \big)
\end{equation}
with 
\begin{equation} \label{eq3.17}
\overbar{K}_{ni}=D_{ni}-\overbar{D}_{ni}=\big( N(U_{ni})-N(\overbar{U}_{ni})\big)-(N(u_n)-N(\tilde{u}_{n})).
\end{equation}
Inserting \eqref{eq3.16} into \eqref{eq3.15} gives 
\begin{equation} \label{eq3.18}
e_{n+1}=\big(I+R_1(hL)hL\big)e_n+h S_n+ \overbar{e}_{n+1},
\end{equation}
where 
\begin{equation} \label{eq3.19}
S_n= R_1(hL)\big( N(u_n)-N(\tilde{u}_{n})\big)+ \sum_{i=2}^{s}b_{i}(h L) \overbar{K}_{ni}.
\end{equation}
Next, we consider $S_n$. Let $\tilde{N}'_n =N'(\tilde{u}_n)$ and $\overbar{N}'_{ni}=N'( \overbar{U}_{ni} )$.
Again, we use the Taylor series expansion of $N(u)$ to get
\begin{subequations} \label{eq3.20}
\begin{align}
N(u_n)-N(\tilde{u}_n)&=\tilde{N}'_n e_n +\tilde{r}_n,  \label{eq3.20a} \\
N(U_{ni})-N( \overbar{U}_{ni} )&=\overbar{N}'_{ni} \overbar{E}_{ni} + \overbar{R}_{ni} \label{eq3.20b}
\end{align}
\end{subequations}
with remainders $\tilde{r}_n$ and $\overbar{R}_{ni}$
\begin{subequations} \label{eq3.21}
\begin{align}
\tilde{r}_n &=\int_{0}^{1} (1-\theta ) N''(\tilde{u}_n + \theta e_n)(e_n, e_n) \text{d}\theta, \label{eq3.21a}\\
\overbar{R}_{ni} &=\int_{0}^{1} (1-\theta ) N''(\overbar{U}_{ni} + \theta \overbar{E}_{ni})(\overbar{E}_{ni},\overbar{E}_{ni})
\text{d}\theta. \label{eq3.21b}
\end{align}
\end{subequations}
Due to Assumption~2 we have
\begin{equation} \label{eq3.22}
\|\tilde{r}_n\|\leq C\|e_n\|^2, \q  \|\overbar{R}_{ni}\|\leq C\|\overbar{E}_{ni}\|^2,
\end{equation}
as long as $e_n$ and $\overbar{E}_{ni}$ remain in a sufficiently small neighborhood of 0.\\
Inserting \eqref{eq3.17} into \eqref{eq3.19} and using \eqref{eq3.20}, we obtain
\begin{equation} \label{eq3.23}
S_n= \big(R_1(hL)- \sum_{i=2}^{s}b_{i}(h L)\Big)(\tilde{N}'_n e_n +\tilde{r}_n)+ \sum_{i=2}^{s}b_{i}(h L) (\overbar{N}'_{ni} \overbar{E}_{ni} + \overbar{R}_{ni}).
\end{equation}
\begin{lemma} \label{lm3.1}
If Assumptions 1 and 2 hold and functions $R_1(hL)$ and $R_{1i}(c_i h L)$ ($  i=2,\ldots,s$) are chosen in such L way that the following bounds 
\begin{eqnarray}  \label{eq3.24}
&\| R_1(hL)\|_{X\leftarrow X} \leq C, \quad \| R_{1i}(c_i h L)\|_{X\leftarrow X} \leq C, \nonumber \\
&\  \|I+ R_{1i}(c_i h L)c_i h L\|_{X\leftarrow X} \leq C,
\end{eqnarray}
hold uniformly, then there exist bounded operators $\mathcal{B}_{n} (e_n)$ on $X$ such that
\begin{equation}  \label{eq3.25}
S_{n}= \mathcal{B}_{n} (e_n)e_n .
\end{equation}
\begin{proof}
First we derive a recursion for $\overbar{E}_{ni}$ by subtracting \eqref{eq3.3a} from \eqref{eq2.5a} and employing \eqref{eq3.20}
\begin{eqnarray} \label{eq3.26}
\overbar{E}_{ni}&=&\big(I+R_{1i}(c_i hL)c_i hL\big)e_n + h \big(c_i R_{1i}(c_i hL)-\sum_{j=2}^{i-1}a_{ij}(h L)\Big) (\tilde{N}'_n e_n +\tilde{r}_n) \nonumber \\
&+&h\sum_{j=2}^{i-1}a_{ij}(h L)(\overbar{N}'_{nj} \overbar{E}_{nj} + \overbar{R}_{nj}). 
\end{eqnarray}
Solving recursion \eqref{eq3.26} by induction on the index $j$ with the help of \eqref{eq3.21} and inserting the obtained result into \eqref{eq3.23} yields \eqref{eq3.25}. The boundedness of $\mathcal{B}_{n} (e_n)$ follows from the assumptions of Lemma~\ref{lm3.1} and the bounds in \eqref{eq3.22}.
\end{proof}
\end{lemma}
We are now ready to provide sufficient conditions for convergence of IMEXP Runge--Kutta methods \eqref{eq2.5}. 
\begin{theorem}\label{th3.2}
Let the initial value problem \eqref{eq1.1} satisfy Assumptions 1--2. Consider for its numerical solution an IMEXP Runge--Kutta method \eqref{eq2.5} that fulfills the order conditions of an exponential Runge--Kutta method up to some order $p$. If coefficients $R_1(h L)$ and $R_{1i}(c_i h L)$ are chosen such that the order conditions in \eqref{eq3.10} are fulfilled and the stability bounds \eqref{eq3.24} and 
\begin{equation}\label{eq3.27}
\|\big(I+R_1 (hL)hL\Big)^{n-j}\|_{X\leftarrow X} \leq C_S, \ j=0,\ldots, n-1
\end{equation}
hold uniformly, then the numerical solution $u_n$ satisfies the error bound
\begin{equation}\label{eq3.28}
\| u_n -u(t_n)\|\leq C h^p
\end{equation}
uniformly on \ $t_0 \leq  t_n =t_0+nh \leq  T$ with a constant $C$ that depends on $T-t_0$, but is independent of $n$ and $h$.
\end{theorem}
\begin{proof}
In view of \eqref{eq3.18} and \eqref{eq3.25}, we get
\begin{equation}  \label{eq3.29}
e_{n+1}=\big(I+R_1(hL)hL\big)e_n+h \mathcal{B}_{n} (e_n)e_n + \overbar{e}_{n+1}.
\end{equation}
Solving recursion \eqref{eq3.29} and using $e_0=0$ finally yields
\begin{equation} \label{eq3.30}
e_{n}=h\sum_{j=0}^{n-1} \big(I+R_1(hL)hL\big)^{n-j} \big( \mathcal{B}_j (e_j)e_j + \frac{1}{h}\overbar{e}_{j+1}\Big).
\end{equation}
 Under the assumptions of Theorem~\ref{th3.2}, we have $\|\mathcal{B}_j (e_j)e_j \|\leq C\|e_j\|$ and $\overbar{e}_{j+1}=\mathcal{O}(h^{p+1})$.
 The proof is completed with the help of the stability bound \eqref{eq3.27} and an application of a discrete Gronwall lemma (see \cite{Em05}) to \eqref{eq3.30}.
\end{proof}
\subsection{On the choice of rational functions $R_1(Z)$ and $R_{1i}(Z)$}\label{sc4.3}
The result of Theorem~\ref{th3.2} shows the sufficient conditions for which the coefficients $R_1(h L)$ and $R_{1i}(c_i h L)$ need to fulfill for an implicit-explicit exponential Runge--Kutta method of order $p$. In particular, the order conditions \eqref{eq3.10} and the stability bounds \eqref{eq3.24}, \eqref{eq3.27} are the constraints for choosing such coefficients. 

First, we focus on conditions \eqref{eq3.10}. Our idea is to use Pad\'e approximation (for instance, see \cite[Chap. IV]{HW96}) for finding a rational function $R(z)$ which Lpproximates to $\varphi_1 (z)$. Given the fact that $\varphi_1 (z)=\dfrac{e^z - 1}{z}$, we employ  Pad\'e approximations to the exponential function to derive the following results 
 \begin{subequations} \label{eq3.31}
\begin{align}
\varphi_1 (z) &=(1-z)^{-1}+\mathcal{O}(z), \label{eq3.31a}\\
\varphi_1 (z) &=(1-\frac{1}{2}z)^{-1}+\mathcal{O}(z^2). \label{eq3.31b}
\end{align}
\end{subequations}
\setlength{\extrarowheight}{4 pt}
\begin{table}[h] 
\caption{Functions $R_1(Z)$ and $R_{1i}(Z)$ and the corresponding exponential Runge-Kutta stiff order conditions for the methods of orders 1 and 2.}
\begin{center}\label{tb1}
\begin{tabular}{ |c|c|c| }
\hline
Order $p$ & $R_1(Z)$ and $R_{1i}(Z)$ & \vtop{\hbox{\strut \q \q Stiff order conditions for}  \hbox{\strut exponential Runge--Kutta methods}}   \\
\hline
1&$R_1(Z)=(1-Z)^{-1}$ &\\
 \hline
2& \vtop{\hbox{\strut $R_{1i}(Z)=(1-Z)^{-1}$ } \hbox{\strut $R_1(Z)= (1-\frac{1}{2}Z)^{-1}$ }}& $\sum_{i=2}^{s}b_{i}(Z)c_i = \varphi_2(Z)$\\
\hline
\end{tabular}
\end{center}
\end{table}
This suggests at once the searching functions $R_1(Z)$ and $R_{1i}(Z)$ for methods of possible orders 1 and 2, see Table~\ref{tb1}. Indeed, we can prove the following result for the case where $X=\mathbb{C}^n$ with the standard inner product denoted by $(\cdot,\cdot )$.
\begin{lemma} \label{lm4.2}
Let $X=\mathbb{C}^n$ and suppose that the matrix $L$ satisfies 
\begin{equation}  \label{eq3.32}
\real(u, Lu)\leq 0 \ \text{for all} \ u \in X.
\end{equation}
Under Assumptions 1 and 2, we have
\begin{equation}  \label{eq3.33}
\|\big((I-hL)^{-1}-\varphi_1 (h L)\big)F(\tilde{u}_n )\|\leq Ch.
\end{equation}
Further assume that $L\dfrac{\dd }{\dd t}N(u(t))\big\lvert_{t=t_n}$ is uniformly bounded on $X$. Then, we have 
\begin{equation}  \label{eq3.34}
\|\big((I-\frac{1}{2}hL)^{-1}-\varphi_1 (h L)\big)F(\tilde{u}_n) \|\leq Ch^2.
\end{equation}
The constants in \eqref{eq3.33} and \eqref{eq3.34} can be chosen uniformly bounded on $[t_0, T]$, and in particular, are independent of $n$ and $h$ (i.e. independent of $\|L\|$).
 \end{lemma}
 \begin{proof}
 Let $\tilde{v}_n=(I-hL)^{-1}F(\tilde{u}_n )$ and $\tilde{w}_n=(I-\frac{1}{2}hL)^{-1}F(\tilde{u}_n )$. Under condition \eqref{eq3.32}, it follows from Theorem 11.2 in \cite[Chap. IV]{HW96} that 
\begin{equation}  \label{eq3.35}
\|\big(I-hL)^{-1}\|\leq \sup_{\real z\leq 0}|\tfrac{1}{1-z}|\leq 1, \ 
\|\big(I-\frac{1}{2}hL)^{-1} \|\leq \sup_{\real z\leq 0}|\tfrac{1}{1-z/2}|\leq 1.
\end{equation}
This reflects the fact that the rational functions $\frac{1}{1-z}$ and $\frac{1}{1-z/2}$ are A-stable.
From this it is easy to see that $\tilde{v}_n, \tilde{w}_n$ are uniformly bounded.
Using $\tilde{u}_{n}'=u'(t_n)=F(\tilde{u}_n )$, one gets
  \begin{subequations} \label{eq3.36}
\begin{align}
\tilde{u}_{n}'&=\tilde{v}_n-hL\tilde{v}_n, \label{eq3.36a}\\
\tilde{u}_{n}'&=\tilde{w}_n-\frac{1}{2}hL\tilde{w}_n, \q
\tilde{u}_{n}''=L\tilde{w}_n-\frac{1}{2}hL^2\tilde{w}_n+ \frac{\dd }{\dd t}N(u(t))\big\lvert_{t=t_n}. \label{eq3.36b}
\end{align}
\end{subequations}
We infer from \eqref{eq3.36} 
that $L\tilde{v}_n, L\tilde{w}_n, L^2\tilde{w}_n$ are uniformly bounded under assumption 2.
We now use the following expansion of $\varphi_1 (h L) F(\tilde{u}_n )$ (see \cite{LO12b}) 
\begin{equation}  \label{eq3.37}
\varphi_1 (h L) F(\tilde{u}_n )=\tilde{u}_{n}'+\frac{h}{2}\big(\tilde{u}_{n}''-2 \varphi_2 (h L) \frac{\dd }{\dd t}N(u(t))\big\lvert_{t=t_n}\big)+\mathcal{O}(h^2),
\end{equation}
where the remainder terms behind the Landau notation (mutiplying by $h^2$) are uniformly bounded by the assumptions 1 and 2.
Inserting \eqref{eq3.36a} into \eqref{eq3.37} one obtains 
\begin{equation}  \label{eq3.38}
\varphi_1 (h L) F(\tilde{u}_n )=\tilde{v}_n-h\big(L\tilde{v}_n-\frac{1}{2}\tilde{u}_{n}''+ \varphi_2 (h L) \frac{\dd }{\dd t}N(u(t))\big\lvert_{t=t_n}\big)+\mathcal{O}(h^2).
\end{equation}
Inserting \eqref{eq3.36b} into \eqref{eq3.37} and using the fact that $\varphi_2 (h L)=\frac{1}{2}I+h\varphi_3 (h L)L$, we end up with
\begin{equation}  \label{eq3.39}
\varphi_1 (h L) F(\tilde{u}_n )=\tilde{w}_n-h^2\big(\frac{1}{4}L^2\tilde{w}_n+ \varphi_3 (h L) L\frac{\dd }{\dd t}N(u(t))\big\lvert_{t=t_n}\big)+\mathcal{O}(h^2).
\end{equation}
It is clear that the coefficients of $h$ in \eqref{eq3.38} and $h^2$ in \eqref{eq3.39} are uniformly bounded under the conditions of Lemma~\ref{lm4.2}. Therefore, one derives at once \eqref{eq3.33} and \eqref{eq3.34}.  
 \end{proof}
\emph{Remark.} Condition \eqref{eq3.32} is often fulfilled for the matrix $L$, which is resulted from the spatial discretization of a second-order strongly elliptic differential operator (e.g. the Laplacian or the gradient). The assumption of uniform boundedness of $L\dfrac{\dd }{\dd t}N(u(t))\big\lvert_{t=t_n}$ is valid for many semilinear parabolic PDEs such as reaction-diffusion equations, the Allen-Cahn equation (see, e.g \cite{Zhu2015}) and the Chafee-Infante problem \cite[Chap.~5]{H81}. 

Next, we verify whether the chosen rational functions in Table~\ref{tb1} satisfy the stability bounds \eqref{eq3.24} and \eqref{eq3.27}. Clearly, under condition \eqref{eq3.32} the first two bounds in  \eqref{eq3.24} are fulfilled due to \eqref{eq3.35}. Similarly, one can see that the third bound (with $R_{1i}(c_i h L)=(I-c_i hL)^{-1}$) and the bound \eqref{eq3.27} (with $R_{1}(hL)=(I-hL)^{-1}$ or $R_{1}(hL)=(I-\frac{1}{2}hL)^{-1}$) are also fulfilled because of the following observations
  \begin{subequations} \label{eq3.40}
\begin{align}
 \|I+ (I-c_i hL)^{-1}c_i h L\|&=\|(I-c_i hL)^{-1}\|\leq 1,\\
 \|\big(I+ (I-hL)^{-1}h L\big)^{n-j}\|&\leq \|(I- hL)^{-1}\|^{n-j}\leq 1,\\
  \|\big(I+ (I-\frac{1}{2}hL)^{-1}hL\Big)^{n-j}\|&\leq \|(I- \frac{1}{2}hL)^{-1} (I+\frac{1}{2}hL)\|^{n-j}\leq 1. \label{eq3.40c}
\end{align}
\end{subequations}
The last inequality in \eqref{eq3.40} holds since $\sup_{\real z\leq 0}|\frac{1+z/2}{1-z/2}|\leq 1$ (the rational function $\frac{1+z/2}{1-z/2}$ is A-stable).
\subsection{Derivation of the first- and second-order methods}
Based on the results of Theorem~\ref{th3.2} and Section~\ref{sc4.3}, it is straightforward to derive methods of orders 1 and 2 with the chosen rational functions in  Table~\ref{tb1} (The convergence of these methods follows directly from Theorem~\ref{th3.2}). In particular,  we obtain the following 1-stage IMEXP Runge--Kutta method of order 1 which we will call $\mathtt{ImExpRK1}$:
\begin{equation}  \label{eq3.41}
u_{n+1}= u_n + h (I-hL)^{-1}F(u_n).
\end{equation}
Using $F(u_n)=Lu_n+ N(u_n)$ and the equality $I+ (I-hL)^{-1}h L=(I-hL)^{-1}$, one realizes that the method $\mathtt{ImExpRK1}$ coincides with the implicit-explicit Rung--Kutta method of order 1: $u_{n+1}= u_n + hLu_{n+1}+hN(u_n).$

As $s=2$ we take $c_2=\frac{1}{2}$ and choose from Table~\ref{tb1} coefficients $R_{12}(hL)= R_1(hL)=(I-\frac{1}{2}hL)^{-1}$, $b_2(hL)=2\varphi_2(hL)$. This leads to the following second-order method which we will call $\mathtt{ImExpRK2}$:
\begin{equation}\label{eq3.42} 
\begin{aligned}
U_{n2}&= u_n + \frac{1}{2} h(I-\frac{1}{2} h L)^{-1}F(u_n), \\
u_{n+1}& = u_n + h (I-\frac{1}{2}h L)^{-1}F(u_n) + 2h\varphi_2(hL)(N(U_{n2})-N(u_n)).
\end{aligned}
\end{equation}
\subsection{A discussion of higher-order methods}
We now discuss whether it is possible to derive IMEXP Runge--Kutta methods of higher orders. Clearly, due to \eqref{eq3.10} such methods require higher-order rational approximations of $\varphi_1 (h L)F(u_n )$, which also need to satisfy the stability bounds \eqref{eq3.24} and \eqref{eq3.27}.
Again, we make use of the Pad\'e approximations to the exponential function to get, for instance, 
 \begin{subequations} \label{eq3.43} 
\begin{align}
\varphi_1 (z) &=(1+\frac{1}{6}z)(1-\frac{1}{3}z)^{-1}+\mathcal{O}(z^3),\\
\varphi_1 (z) &=(1-\frac{1}{2}z+ \frac{1}{12}z^2)^{-1}+\mathcal{O}(z^4).
\end{align}
\end{subequations}
This offers, for example, the choice of $R_1(hL)=(1+\frac{1}{6}hL)(1-\frac{1}{3}hL)^{-1}$ for methods of order 3. However, we can show that such a method requires much stronger regularity assumptions in order to fulfill \eqref{eq3.10}.  Construction of higher order methods will be the subject of future publications.
\section{Hybrid IMEXP}
In this section, we consider the case where the stiffness of problem \eqref{eq1.1} comes from both the linear $L$ and nonlinear $N(u)$ operators. 
Motivated by the results of Section~\ref{sc4}, we propose the following two perturbation schemes of \eqref{eq3.42}
\begin{subequations}\label{eq3.44} 
\begin{align}
U_{n2}&= u_n + \frac{1}{2} h(I-\frac{1}{2} h L)^{-1}F(u_n),\label{eq3.44a} \\
u_{n+1}& = u_n + h (I-\frac{1}{2}h L)^{-1}F(u_n) + 2h\varphi_2(h J_n)(N(U_{n2})-N(u_n)) \label{eq3.44b}
\end{align}
\end{subequations}
and 
\begin{equation}\label{eq3.45} 
\begin{aligned}
U_{n2}&= u_n + \frac{1}{2} h(I-\frac{1}{2} h L)^{-1}F(u_n), \\
u_{n+1}& = u_n + h (I-\frac{1}{2}h L)^{-1}F(u_n) + 2h\varphi_2(h N'_n)(N(U_{n2})-N(u_n)),
\end{aligned}
\end{equation}
which will be called $\mathtt{HImExp2J}$ and  $\mathtt{HImExp2N}$, respectively.
In the following we will show that these two schemes are of order 2 as well. First, note that since we are working with the nonlinearity $N(u)$ which is assumed to have some stiffness, the norms of the Jacobian $N'(u)$ and its higher derivatives could be large so that Assumption~2 is not feasible. Thus, in order to give a convergence result for those two schemes, we will make use of the following much weaker, but reasonable, assumption on the solution and the nonlinearity.

{\em Assumption 3. Suppose that \eqref{eq1.1} possesses a solution that is thrice differentiable with (uniformly) bounded derivatives in $X$, and that  $N: X \rightarrow X$ is twice continuously Fr\'echet differentiable in a strip along the exact solution. } 

Under this assumption one can see that $N(u)$ also satisfies the Lipschitz condition in a strip along the exact solution. Moreover, in the case where $X=\mathbb{C}^n$ we have the following property of the second-order Fr\'echet derivative 
\begin{equation}\label{eq3.46}
N(u+\Delta u) = N(u) +  N'(u)\Delta u + \int_{0}^{1} (1-s)N''(u+s\Delta u)(\Delta u,\Delta u)\dd s.
\end{equation}
In the following, and unless otherwise specified, we will work in $X=\mathbb{C}^n$.

Since $N'_n=N'(u_n)$ is a linear bounded operator on $X$, the Jacobians $J_n=L+N'_n$ and $N'_n$ generate strongly continuous semigroups  $\ee^{tJ_n}$ and $\ee^{tN'_n}$(see \cite[Chap.~3.1]{PAZY83}).  Thus one infers from Assumption~1 that $\ee^{hJ_n}$ and $\ee^{hN'_n}$ are uniformly bounded, and thus so are $\varphi_2(h J_n)$ and $\varphi_2(h N'_n)$.
\subsection{Expansion of the exact solution}
Using Assumptions 1 and 3, we will derive an expansion of the exact solution of \eqref{eq1.1} at time $t_{n+1}$, i.e., $u(t_{n+1})$. Again, let $\tilde{u}_n=u(t_n)$. Expressing $u(t_{n+1})$ by the variation-of-constants formula gives
\begin{equation} \label{eq3.48}
\tilde{u}_{n+1}=u(t_{n+1})=\ee^{h L}\tilde{u}_n+h \int_{0}^{1} \ee^{(1-\theta )h L} N(u(t_n +\theta h))\,\text{d}\theta.
\vspace{-1.5mm}
\end{equation}
Now let $\mathcal{U}_n=u(t_n+\theta h)-\tilde{u}_n$ and let  $\tilde{u}'_n, \tilde{u}''_n, \tilde{u}'''_n$ denote the first, second, and third derivative of the exact solution  $u(t)$ of \eqref{eq1.1} 
evaluated at time $t_n$. 
Under Assumption~3 we can expand $u(t_n+\theta h)$ in a Taylor series at $t_n$ to get 
\begin{equation}\label{eq3.49} 
\mathcal{U}_n= \theta h \tilde{u}'_n +  \theta^2 h^2 \int_{0}^{1} (1-s)u''(t_n+\theta h s) \dd s. 
\end{equation}
Using \eqref{eq3.46} with $ u=\tilde{u}_n, \Delta u= \mathcal{U}_n $, one derives 
\begin{equation}\label{eq3.50} 
N(u(t_n +\theta h))= N(\tilde{u}_n) +  N'(\tilde{u}_n)\mathcal{U}_n + \int_{0}^{1} (1-s)N''(\tilde{u}_n+s\mathcal{U}_n)(\mathcal{U}_n,\mathcal{U}_n)\dd s.\\
\end{equation}
With the help of \eqref{eq3.49}, inserting \eqref{eq3.50} into \eqref{eq3.48} and using \eqref{eq2.2} (with $k=1, 2$), we eventually obtain an expansion of the exact solution 
\begin{equation}\label{eq3.51}
\tilde{u}_{n+1}=\tilde{u}_n+h \varphi_1(hL)F(\tilde{u}_n)+h^2 \varphi_2(h L)\tfrac{\dd }{\dd t}N(u(t))\big\lvert_{t=t_n}+ h^3 \mathcal{R}_n
\end{equation}
 with 
 \begin{equation}\label{eq3.52}
 \mathcal{R}_n= \int_{0}^{1} \ee^{(1-\theta )h L} \int_{0}^{1} \theta^2(1-s)\big( N''(\tilde{u}_n)u''(t_n+\theta h) +N''(\tilde{u}_n+s\mathcal{U}_n)(\mathcal{V}_n,\mathcal{V}_n) \Big)\dd s \dd\theta,
 \end{equation}
 where $\mathcal{V}_n=\tilde{u}'_n +  \theta h \int_{0}^{1} (1-s)u''(t_n+\theta h s) \dd s $. Note that this remainder is not the same as the one in the expansion of $u(t_{n+1})$ given in \cite{LO14b}, which requires Assumption~2. It is clear that $\|\mathcal{R}_n\|\leq C$ (uniformly) due to Assumptions 1 and 3.
\subsection{Local error analysis of schemes  $\mathtt{HImExp2J}$ and  $\mathtt{HImExp2N}$}
Using the expansion of the exact solution in \eqref{eq3.51}, we prove the following result concerning the local errors of schemes  \eqref{eq3.44} and \eqref{eq3.45} at time $t=t_{n+1}$.
\begin{lemma} \label{lm5.1}
Let the initial value problem \eqref{eq1.1} satisfy  Assumptions 1 and 3. Further assume that $L$ satisfies the condition \eqref{eq3.32} and that $L\tfrac{\dd }{\dd t}N(u(t))\big\lvert_{t=t_n}$ is uniformly bounded on $X$. Then, both schemes \eqref{eq3.44} and \eqref{eq3.45} have order of consistency 3.
\end{lemma}
 \begin{proof}
 First, we study the local error of scheme \eqref{eq3.44}, i.e., $\mathtt{HImExp2J}$. For that we consider one step with initial value $\tilde{u}_n$, i.e.
 \begin{subequations}\label{eq3.53} 
\begin{align}
\overbar{U}_{n2}&= \tilde{u}_n+ \frac{1}{2} h(I-\frac{1}{2} h L)^{-1}F(\tilde{u}_n), \label{eq3.53a}  \\
\overbar{u}_{n+1}& =\tilde{u}_n + h (I-\frac{1}{2}h L)^{-1}F(\tilde{u}_n) + 2h\varphi_2(h \tilde{J}_n)(N(\overbar{U}_{n2})-N(\tilde{u}_n)). \label{eq3.53b}
\end{align}
\end{subequations}
Therefore, the local error of  \eqref{eq3.44}  at time $t_{n+1}$ is given by 
 \begin{equation}\label{eq3.53add} 
\bar{e}_{n+1,\mathtt{HImExp2J}} =\overbar{u}_{n+1}- \tilde{u}_{n+1}.
\end{equation}
Proceeding in the same manner as in Lemma 3.4.1 of \cite{LO12b}, we get 
 \begin{eqnarray}  \label{eq3.54}
\varphi_1 (h L) F(\tilde{u}_n )&=&\tilde{u}_{n}'+\frac{h}{2}\big(\tilde{u}_{n}''-2 \varphi_2 (h L) \tfrac{\dd }{\dd t}N(u(t))\big\lvert_{t=t_n}\big) \nonumber \\
&+&h^2 \varphi_3 (h L) (\tilde{u}_{n}'''- \tfrac{\dd^2 }{\dd t^2}N(u(t))\big\lvert_{t=t_n}).
\end{eqnarray}
 Clearly, due to Assumptions 1 and 3, we get back \eqref{eq3.37} from \eqref{eq3.54}. This shows that the result of Lemma~\ref{lm4.2} holds. Using the same notation $\tilde{w}_n=(I-\frac{1}{2}hL)^{-1}F(\tilde{u}_n )$, \eqref{eq3.53} can be rewritten as
  \begin{equation}\label{eq3.55} 
\overbar{u}_{n+1} =\tilde{u}_n + h \tilde{w}_n + 2h\varphi_2(h \tilde{J}_n)(N(\tilde{u}_n+ \frac{1}{2} h\tilde{w}_n)-N(\tilde{u}_n)).
\end{equation}
Employing \eqref{eq3.39}, applying \eqref{eq3.46} to $N(\tilde{u}_n+ \tfrac{1}{2} h\tilde{w}_n)$, and noting that $\tilde{w}_n=\tilde{u}'_n+ \tfrac{1}{2}h L\tilde{w}_n$ (see \eqref{eq3.36b}) one finally gets
  \begin{equation}\label{eq3.56} 
\overbar{u}_{n+1} =\tilde{u}_n + h \varphi_1(hL)F(\tilde{u}_n)+h^2 \varphi_2(h \tilde{J}_n)\tfrac{\dd }{\dd t}N(u(t))\big\lvert_{t=t_n} +h^3 \mathcal{\overbar{R}}_n
\end{equation}
with 
 \begin{equation}\label{eq3.57} 
 \begin{aligned}
 \mathcal{\overbar{R}}_n &=\frac{1}{4}L^2 \tilde{w}_n+ \varphi_3 (h L) L\frac{\dd }{\dd t}N(u(t))\big\lvert_{t=t_n}+\varphi_3 (h L) (\tilde{u}_{n}'''- \frac{\dd^2 }{\dd t^2}N(u(t))\big\lvert_{t=t_n})\\
&+ \frac{1}{2}\varphi_2 (h \tilde{J}_n) L\tilde{w}_n+\frac{1}{4}\int_{0}^{1} (1-s) N''(\tilde{u}_{n}+\frac{1}{2}sh\tilde{w}_n)(\tilde{w}_n,\tilde{w}_n)\dd s,
 \end{aligned}
 \end{equation}
 which is uniformly bounded due to the assumptions of  Lemma~\ref{th5.1}.
Inserting \eqref{eq3.51} and \eqref{eq3.56} into  \eqref{eq3.53add} gives
 \begin{equation} \label{eq3.58}
\bar{e}_{n+1,\mathtt{HImExp2J}} =h^2 (\varphi_2(h \tilde{J}_n)-\varphi_2(h L)) \frac{\dd }{\dd t}N(u(t))\big\lvert_{t=t_n} +h^3 (\mathcal{\overbar{R}}_n- \mathcal{R}_n).
\end{equation}
Since $\tilde{J}_n=L+\tilde{N}'_n$ (here $\tilde{N}'_n=N'(\tilde{u}_n)$), it is easy to show by using \eqref{eq2.2} and applying the variation-of-constants formula to the differential equation $v'(t)=\tilde{J}_n v(t)=Lv(t)+ \tilde{N}'_n v(t), \ v(0)=I$ that
 \begin{equation} \label{eq3.59}
\varphi_2(h \tilde{J}_n)-\varphi_2(h L)=h\phi (hL,h\tilde{N}'_n),
\end{equation}
where
 \begin{equation} \label{eq3.60}
\phi (hL,h\tilde{N}'_n)=\int_{0}^{1}\int_{0}^{1} \ee^{(1-s)(1-\theta )h L} \theta(1-\theta)\tilde{N}'_n \ee^{s(1-\theta)h \tilde{J}_n} \dd s \dd \theta
\end{equation}
which is a bounded operator.
We now insert \eqref{eq3.59} into \eqref{eq3.58} to obtain 
 \begin{equation} \label{eq3.61}
\bar{e}_{n+1,\mathtt{HImExp2J}} =h^3 \big(\phi (hL,h\tilde{N}'_n) \frac{\dd }{\dd t}N(u(t))\big\lvert_{t=t_n} +\mathcal{\overbar{R}}_n- \mathcal{R}_n \big)=\mathcal{O}(h^3).
\end{equation}

Next, we consider the local error of scheme \eqref{eq3.45}, i.e., $\mathtt{HImExp2N}$. Let denote the local error of  it at time $t_{n+1}$ by $\bar{e}_{n+1,\mathtt{HImExp2N}}$. It can be seen that one can analyze this local error in a very similar way as done for scheme $\mathtt{HImExp2J}$. Thus we only focus on the following new aspects. Instead of \eqref{eq3.58}, we now get 
\begin{equation} \label{eq3.62}
\bar{e}_{n+1,\mathtt{HImExp2N}} =h^2 (\varphi_2(h \tilde{N}'_n)-\varphi_2(h L)) \frac{\dd }{\dd t}N(u(t))\big\lvert_{t=t_n} +h^3 (\mathcal{\overbar{R}}_n- \mathcal{R}_n)
\end{equation}
with $\tilde{N}'_n$ in place of $\tilde{J}_n$ appearing in \eqref{eq3.58} as well as in \eqref{eq3.57} (for $\mathcal{\overbar{R}}_n$). Using the observation 
 \begin{equation} \label{eq3.63}
\varphi_2(h \tilde{N}'_n)-\varphi_2(h L)=\big(\varphi_2(h \tilde{N}'_n)-\varphi_2(h \tilde{J}_n)\big)+ \big(\varphi_2(h \tilde{J}_n)-\varphi_2(h L)\big),
\end{equation}
we can show, by again using \eqref{eq2.2}, applying the variation-of-constants formula to the differential equation $y'(t)=\tilde{N}'_n y(t)=\tilde{J}_n y(t)-Ly(t), \ y(0)=I$, and employing the fact that $\ee^z=1+z\varphi_1(z)$,  that
 \begin{equation} \label{eq3.64}
\varphi_2(h \tilde{N}'_n)-\varphi_2(h L)=h\Big(\psi_1 (h\tilde{J}_n)L-h\psi_2 (h\tilde{J}_n, h\tilde{N}'_n) +\phi (hL,h\tilde{N}'_n)\Big).
\end{equation}
Here, $\phi (hL,h\tilde{N}'_n)$ is given in \eqref{eq3.60} and 
 \begin{equation} \label{eq3.65}
\psi_1 (h\tilde{J}_n)=\int_{0}^{1} \int_{0}^{1}\ee^{(1-\theta )h\tilde{J}_n} \theta(\theta-1) \dd s \dd \theta
\end{equation}
which is also a bounded operator, and 
 \begin{equation} \label{eq3.65add}
\psi_2 (h\tilde{J}_n, h\tilde{N}'_n)=\int_{0}^{1} \int_{0}^{1}\ee^{(1-s)(1-\theta )h\tilde{J}_n} s\theta(1-\theta)^2 L\tilde{N}'_n \varphi_1(s (1-\theta)h \tilde{N}'_n) \dd s   \dd \theta
\end{equation}
is bounded as well (due to the assumption $L\tfrac{\dd }{\dd t}N(u(t))\big\lvert_{t=t_n}= L\tilde{N}'_n u'(t_n)$ is uniformly bounded).

Inserting \eqref{eq3.64} into \eqref{eq3.62} clearly shows that
\begin{equation} \label{eq3.66}
\bar{e}_{n+1,\mathtt{HImExp2N}}=h^3 \big(\psi_1 (h\tilde{J}_n)L\tfrac{\dd }{\dd t}N(u(t))\big\lvert_{t=t_n} \big)+ \mathcal{O}(h^3)=\mathcal{O}(h^3).
\end{equation}
\end{proof}
We are now ready to state the main result of this section.
\begin{theorem}\label{th5.1}
Let the initial value problem \eqref{eq1.1} satisfy  the conditions of Lemma~\ref{lm5.1}. Then, the numerical solution $u_n$ of the hybrid implicit-explicit exponential methods \ $\mathtt{HImExp2J}$ or  $\mathtt{HImExp2N}$ satisfies the error bound
\begin{equation}\label{eq3.67}
\| u_n -u(t_n)\|\leq C h^2
\end{equation}
uniformly on \ $t_0 \leq  t_n =t_0+nh \leq  T$ with a constant $C$ that depends on $T-t_0$, but is independent of $n$ and $h$.
\end{theorem}

\begin{proof}
First, we prove the convergence result for scheme $\mathtt{HImExp2J}$. 
Let $ e_{n+1} = u_{n+1} - u(t_{n+1})=u_{n+1} - \tilde{u}_{n+1}$ denote the global error of scheme  \eqref{eq3.44} at time $t_{n+1}$. We have
\begin{equation} \label{eq3.68}
e_{n+1}=u_{n+1}-\bar{u}_{n+1}+ \bar{e}_{n+1,\mathtt{HImExp2J}}.
\end{equation}
Subtracting \eqref{eq3.53b} from \eqref{eq3.44b}, inserting the obtained result into \eqref{eq3.68}, and using $F(u)=Lu+N(u)$ shows that
\begin{equation} \label{eq3.69}
e_{n+1}=R(hL)e_n + h T_n+\bar{e}_{n+1,\mathtt{HImExp2J}}
\end{equation}
with $R(hA)=I+ (I-\frac{1}{2}hL)^{-1}hL=(I- \frac{1}{2}hL)^{-1} (I+\frac{1}{2}hL)$
and 
\begin{equation} \label{eq3.70}
\begin{aligned}
T_n=&((I- \frac{1}{2}hL)^{-1} -2 \varphi_2(h \tilde{J}_n))(N(u_n)-N(\tilde{u}_n))+2 \varphi_2(h J_n)(N(U_{n2})-N(\overbar{U}_{n2}))\\
+& 2 (\varphi_2(h J_n) -\varphi_2(h \tilde{J}_n))(N(\overbar{U}_{n2})-N(u_n)).
\end{aligned}
\end{equation}
Subtracting \eqref{eq3.53a} from \eqref{eq3.44a} and using the identity $I+ \frac{1}{2}hL(I- \frac{1}{2}hL)^{-1}=(I- \frac{1}{2}hL)^{-1}$ gives
\begin{equation} \label{eq3.71}
U_{n2}-\overbar{U}_{n2}=(I- \frac{1}{2}hL)^{-1}e_n+ \frac{1}{2}h(I- \frac{1}{2}hL)^{-1}(N(u_n)-N(\tilde{u}_n)).
\end{equation}
Using the Lipschitz property of $N(u)$ and employing the bound $\|\varphi_2(h J_n) -\varphi_2(h \tilde{J}_n)\|\leq Ch\|e_n\|$ (as a consequence of Lemma 2.4.3 in \cite{LO14a}), we derive at once $\|T_n\|\leq C \|e_n\|$ as long as the global errors $e_n$ remain in a sufficiently small neighborhood of 0.
Solving recursion \eqref{eq3.69} and using $e_0=0$ finally yields
\begin{equation} \label{eq3.72}
e_{n}=h\sum_{j=0}^{n-1} (R(hL))^{n-j} \big(T_j  + \frac{1}{h}\bar{e}_{n+1,\mathtt{HImExp2J}}\Big).
\end{equation}
Now using  \eqref{eq3.61} and the bound  \eqref{eq3.40c} ($\| R(hL))^{n-j}\| \leq 1$) we can estimate 
\begin{equation} \label{eq3.73}
\|e_n\| \leq C h\sum_{j=0}^{n-1} (\|e_j\|+h^2).
\end{equation}
Again, an application of a discrete Gronwall lemma to \eqref{eq3.73} (for instance, see \cite{Em05}) shows the desired bound \eqref{eq3.67}.

Next, because of the similarity between the structure of the two schemes $\mathtt{HImExp2N}$ and $\mathtt{HImExp2J}$, the convergence proof for $\mathtt{HImExp2N}$  method can be carried out in a very similar way. We thus obtain a representation of the global error of scheme \eqref{eq3.45} at time $t_{n}$ as
\begin{equation} \label{eq3.74}
e_{n}=h\sum_{j=0}^{n-1} (R(hL))^{n-j} \big(T_j  + \frac{1}{h}\bar{e}_{n+1,\mathtt{HImExp2N}} \Big).
\end{equation}
However, here $T_j$ is not the same as the one in  \eqref{eq3.70}. In fact it is given by
\begin{equation} \label{eq3.75}
\begin{aligned}
T_j=&((I- \frac{1}{2}hL)^{-1} -2 \varphi_2(h \tilde{N}_j))(N(u_j)-N(\tilde{u}_j))+2 \varphi_2(h N_j)(N(U_{j2})-N(\overbar{U}_{j2}))\\
+& 2 (\varphi_2(h N_j) -\varphi_2(h \tilde{N}_j))(N(\overbar{U}_{j2})-N(u_j)).
\end{aligned}
\end{equation}
In this case one can show that $\| T_j\|\leq C\|e_j\|+ Ch^2$ (by using the boundedness of $(I- \frac{1}{2}hL)^{-1} -2 \varphi_2(h \tilde{N}_j), \ \varphi_2(h N_j)$, the Lipschitz property of $N(u)$, and the estimate $\|\varphi_2(h N_j) -\varphi_2(h \tilde{N}_j\|\leq Ch\|N_j-\tilde{N}_j\|\leq Ch$, which can be easily proved by applying Lemma~2.4.2 given in  \cite{LO14a}). Using \eqref{eq3.66} (which implies $\|\bar{e}_{n+1,\mathtt{HImExp2N}}\|\leq Ch^3$) and the bound  \eqref{eq3.40c}, i.e. $\| R(hL))^{n-j}\| \leq 1$, we again get the same bound for the global error of scheme $\mathtt{HImExp2N}$ as in \eqref{eq3.73} which proves the bound \eqref{eq3.67}.

\end{proof}


\section{Numerical Experiments}\label{sec:numericalExperiments}

\subsection{Integrators and implementation}
In this section we verify theoretical predictions regarding the accuracy and efficiency of the newly constructed IMEXP schemes { \tt ImExpRK2} \eqref{eq3.42}, { \tt HImExp2J} \eqref{eq3.44}, and {\tt HImExp2N} \eqref{eq3.45}.  Several IMEX schemes were compared in \cite{Ruuth} and it was shown that the second-order semi-implicit backwards differentiation formula-based IMEX scheme {\tt 2-sBDF} has the least stringent stability restrictions on the time step and consequently the best computational efficiency for reaction-diffusion problems arising in pattern formation.  Same conclusion was reached in \cite{SBL12}.  We thus choose  to compare the new IMEXP schemes with the widely used {\tt 2-sBDF} integrator \cite{Ruuth}:
\begin{equation}
3U_{n+1} - 4U_n - U_{n-1} = 2h(LU_{n+1} + 2N(U_n) - N(U_{n-1}).
\end{equation}
Clearly since operator $N$ is treated explicitly, there are stability restrictions on the time step of {\tt 2-sBDF} associated with the degree of stiffness in $N$. Since the stability constraints on the time step of the new IMEXP schemes are less restrictive, we expect to be able to take a larger time step compared to {\tt 2-sBDF}. However, treating $N$ exponentially imposes additional computational cost as compared with {\tt 2-sBDF}.  Below we present numerical experiments that demonstrate that as the stiffness of $N$ is increased the ability to take a significantly larger time step outweighs the additional computational cost per time step and the IMEXP schemes can outperform {\tt 2-sBDF}.

All integrators were implemented in MATLAB. The adaptive Krylov algorithm as described in \cite{TLP12} was used to compute products of matrix  $\varphi$-functions and vectors. MATLAB's built-in generalized minimal residual method {\tt gmres} function was used to compute the implicit terms with a tolerance set at$10^{-12}$ for all problems.  For demonstration and comparison purposes, we chose to use MATLAB's preconditioned {\tt gmres} with a sparse incomplete Cholesky factorization ({\tt ichol}) as the preconditioner.   The default parameter values were used except in cases where a given tolerance was desired.  All deviations from the default parameter values are detailed in descriptions of the individual problems below. For each problem, the same set of parameters were used for all integrators.  All simulations were performed with a constant time step starting with $h_1$ and the time step sizes were halved $h_i=h_1/2^{i-1},$ $i=2,\dots,5$ to generate the graphs. Table (\ref{table:timestepSize}) shows the largest time step size taken for each of the problems and each of the methods.

\subsection{Test problems and verification of accuracy}
We choose the following three test problems that satisfy Assumptions 1-3 from the previous section.  Note that the first problem, the 1D semilinear parabolic equations from \cite{HO05}, was originally designed to demonstrate the order reduction non-stiffly accurate methods can experience when a problem is very stiff. However, this problem is not very computationally intensive and thus we only include it to verify the accuracy of our method and do not use it for performance analysis.  

%

\noindent \emph{1D Semilinear parabolic.} One-dimensional semilinear parabolic problem \cite{HO05}
					\begin{equation*}
					\frac{\partial u}{\partial t}(x,t)-\frac{\partial^2 u}{\partial x^2}(x,t)=\int_0^1 u(x,t)dx + \Phi(x,t), \quad x\in [0,1], \quad t\in [0,1]
					\end{equation*}		
		  with homogeneous Dirichlet boundary conditions.  The source function $\Phi$ is chosen so that $u(x,t)=x(1-x)e^t$ is the exact solution.  
To achieve the desired tolerance we increased the default maximum number of GMRES iterations in MATLAB routines to 500.

  \noindent \emph{Allen-Cahn 2D.} Two-dimensional stiff Allen-Cahn equation with periodic boundary conditions \cite{Zhu2015}:
		$$\begin{aligned}
&\frac{\partial u}{\partial t}=\Delta u - \frac{1}{\epsilon^2}(u^3-u),\quad x\in [-0.5,0.5]^2, \quad t\in [0,0.075]\\
		&u(0,x)=\tanh\left(\frac{R_0-\|{x}_2\|}{\sqrt{2}\epsilon}\right),
		\end{aligned},$$
		where   $R_0=0.4$ and  $\epsilon$ specified to be $0.01, 0.02,$ and $0.005$. 

\noindent\emph{
		Schnakenberg 2D.}  Two-dimensional Schnakenberg system \cite{schnakenberg, SBL12}
        			\begin{equation*} \label{eqn:schnacken}
			\begin{array}{l}
			\dfrac{\partial u}{\partial t} = \gamma(a-u+u^2v)+\Delta u,\\
			\dfrac{\partial v}{\partial t}= \gamma(b-u^2v)+d\Delta v,\quad (x,y) \in [0,1]^2
			\end{array}
			\end{equation*}
			with homogeneous Neumann boundary conditions, $a=0.1$, $b=0.9$, $d=10$, and  $\gamma = 1000 \;\&\; \gamma= 10000$.  The initial conditions were chosen to be perturbations of the equilibrium $(\bar{u},\bar{v})=(a+b,b/(a+b)^2)$ as in \cite{madzvamuse2007,SBL12}.
Form this problem the default restart parameter for the MATLAB's GMRES routine was changed from 10 to 20. 

 The Laplacian term $\Delta$ in all problems was discretized using the standard second order finite differences.  For 2D Allen-Cahn equation we use 150 nodes while for 2D Schnackenberg's system 128 spatial discretization points in each spatial dimension.  Except for the semilinear parabolic problem where an exact solution is provided, a reference solution was computed for the remaining problems using MATLAB's {\tt ode15s} integrator with absolute and relative tolerances set to $10^{-14}$. The error was defined as the discrete infinity (maximum) norm of the difference between the computed and the reference solutions.  Figure \ref{fig:loglog1} shows the order attained by all methods for each of the problems.  For convenience we included a line of slope two (dotted) in the graphs.  As can be seen from this plot all methods achieve the second order of accuracy as predicted by the theory. 
 
  \begin{table}[H]
\centering
\caption{Largest time step sizes taken for each of the problems and each of the methods.}
\label{table:timestepSize}
\begin{tabular}{lcccccc}
                 & \multicolumn{3}{c}{2D Schnakenberg $N=128^2$} & \multicolumn{3}{c}{2D Allen-Cahn $N=150^2$} \\
                 & $\gamma=1000$      & $\gamma=10000$     & $\gamma=50000$    & $\epsilon=0.02$     & $\epsilon=0.01$    & $\epsilon=0.005$  \\ \hline
{\tt HImExp2J} & $1\cdot 10^{-2}$   & $1\cdot 10^{-3}$   & $5\cdot 10^{-5}$  & $5\cdot 10^{-4}$   & $2\cdot 10^{-4}$  &      $5\cdot 10^{-6} $           \\
{\tt HImExp2N} & $1\cdot 10^{-3}$   & $1\cdot 10^{-4}$   & $5\cdot 10^{-6}$  & $1\cdot 10^{-4}$   & $5\cdot 10^{-5}$  &         $1\cdot 10^{-6}$         \\
{\tt ImExpRK2}  & $1\cdot 10^{-3}$   &       $1\cdot 10^{-4}$          &       $1\cdot 10^{-6}$        & $2\cdot 10^{-4}$   & $5\cdot 10^{-5}$  &             $1\cdot 10^{-5}$     \\
{\tt 2-sBDF}   & $5\cdot 10^{-4}$   & $1\cdot 10^{-5}$   & $5\cdot 10^{-6}$  & $2\cdot 10^{-4}$   & $5\cdot 10^{-5}$  &         $5\cdot 10^{-7}$        
\end{tabular}
\end{table}

 \begin{figure}[H]
\centering
\caption{Log-log plots of the error vs. time step size. For convenience a line with slope two (dotted) is shown.}
    \label{fig:loglog1}
\begin{tabular}{C{.5\textwidth}R{3pt}C{.5\textwidth}}
 (a) Semilinear parabolic $N=500$ \newline  [1.3ex]
 {\includegraphics[width=2.5in, keepaspectratio]{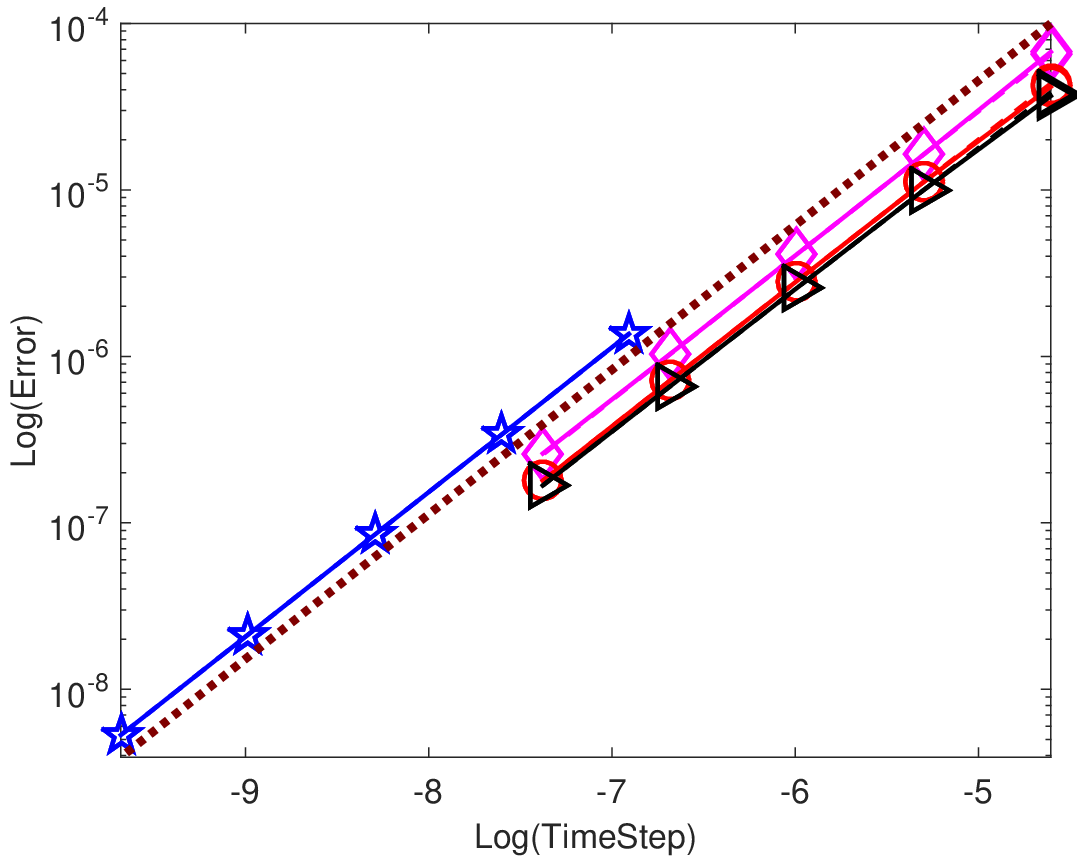}} & &
{\includegraphics[scale=0.18]{Legend} }   \\
\multicolumn{3}{c}{2D Schnakenberg $N=128^2$} \\[-0.5ex]
(b) $\gamma =10^3$ \newline [0.8ex]
{\includegraphics[width=2.5in, keepaspectratio]{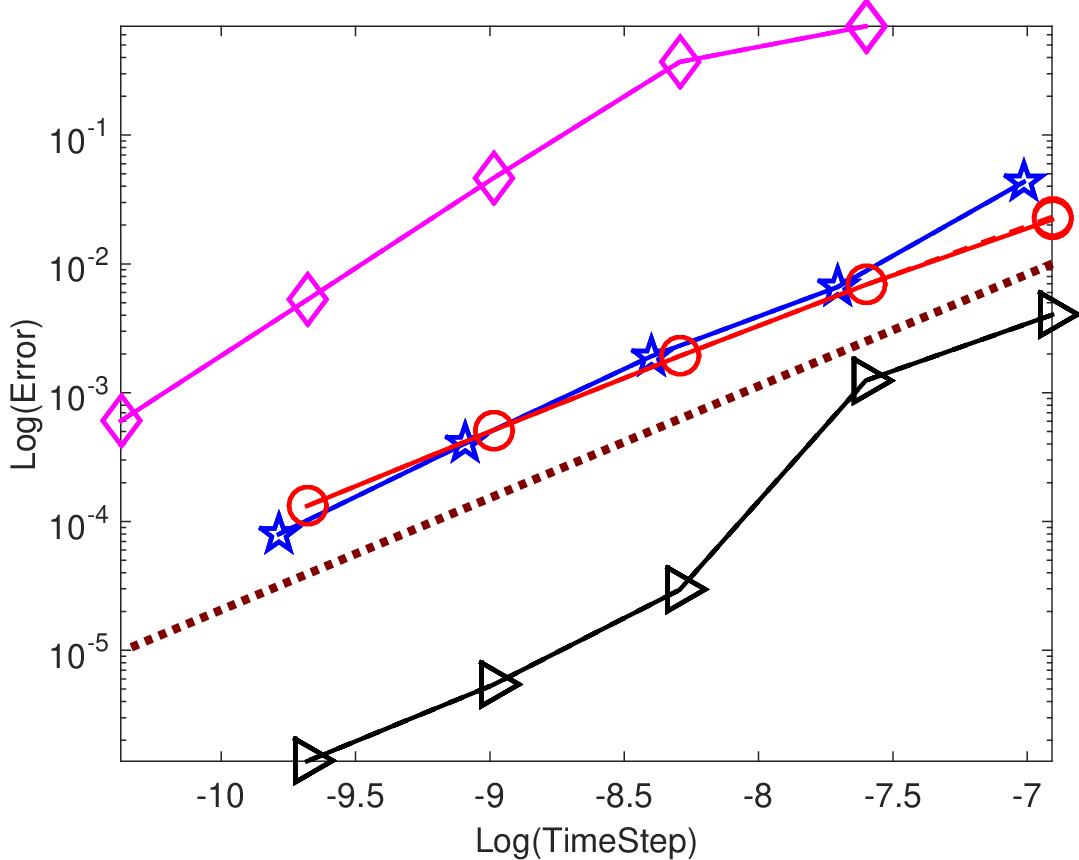} }  &  &
(c)  $ \gamma =10^4$  \newline [1.6ex]
{\includegraphics[width=2.5in, keepaspectratio]{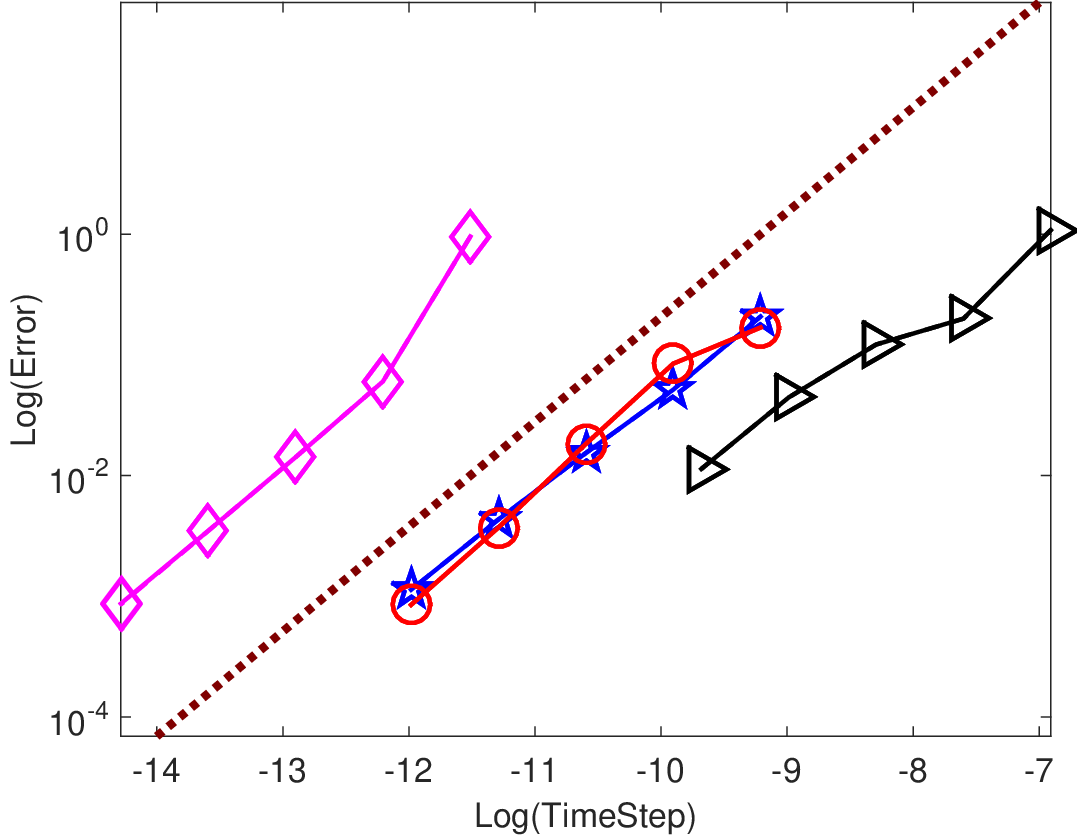}} \\
\multicolumn{3}{c}{2D Allen-Cahn  $ N=150^2$}\\[-0.5ex] 
(d)  $\epsilon=0.02$\newline [0.8ex]
{\includegraphics[width=2.5in, keepaspectratio]{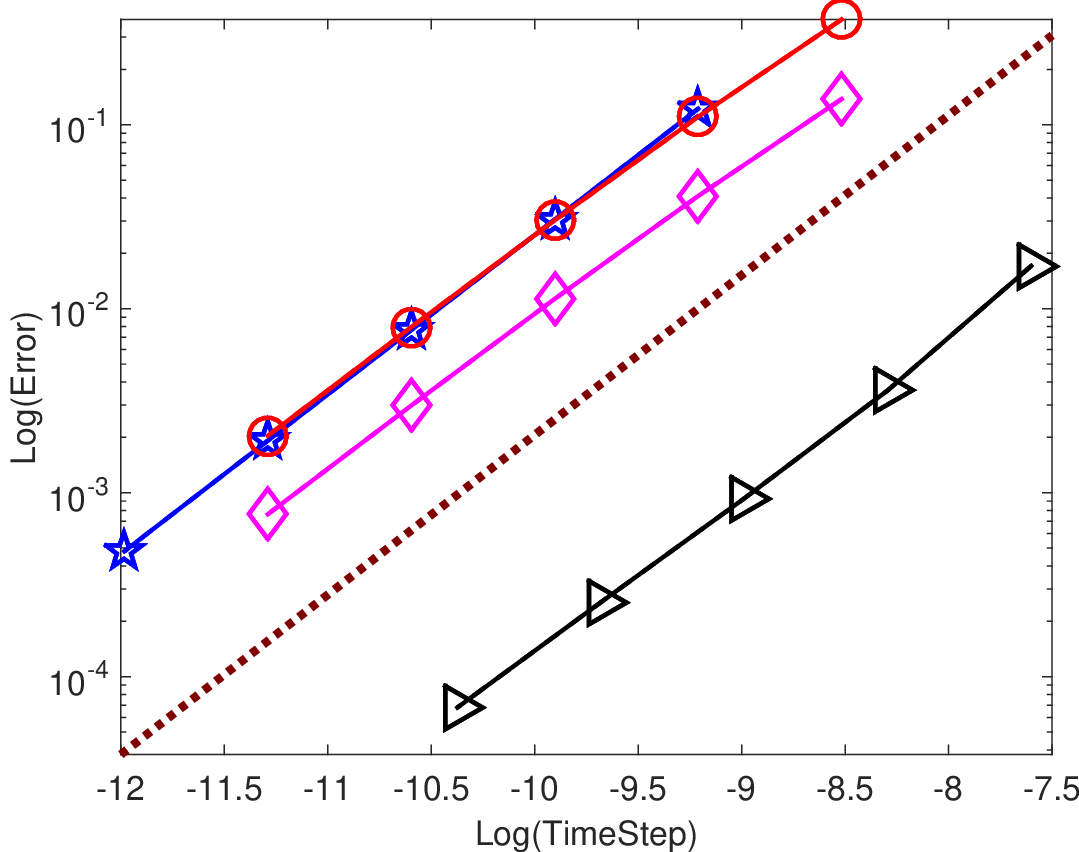} }  &  & 
(e)  $\epsilon=0.01$  \newline [1.6ex]
 {\includegraphics[width=2.5in, keepaspectratio]{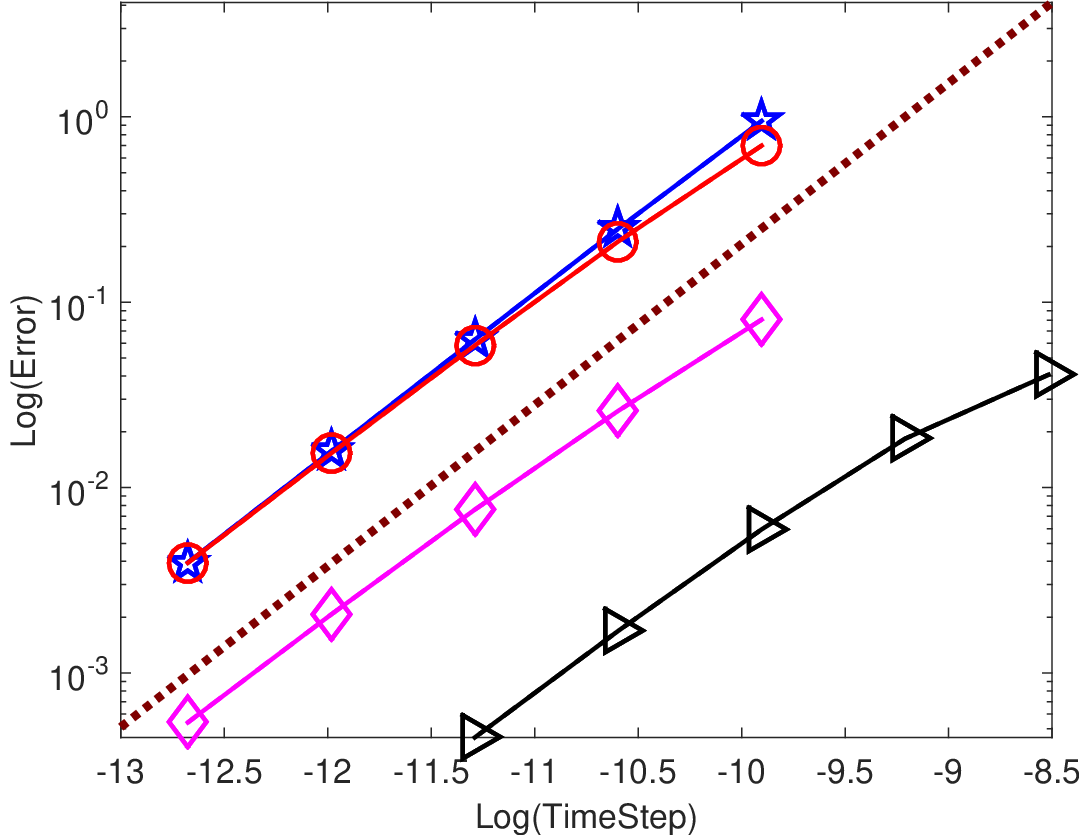} }                   
\end{tabular}
\end{figure}


\subsection{Performance evaluation}

Figures \ref{fig:prec_AllenCahn} and \ref{fig:precSchnackenberg} show the precision diagrams for all four methods for the 2D Allen-Cahn equation and the 2D Schnackenberg model respectively.  For Allen-Cahn equation the stiffness of the nonlinear operator $N$ is increased with the decreasing value of the $\epsilon$ parameter. We choose $\epsilon= 0.02$, 0.01 and 0.005, with the most stiff problem corresponding to $\epsilon= 0.005$.    The nonlinear operator of the 2D Schnakenberg's system becomes more stiff as the parameter $\gamma$ is increased. We set $\gamma = 10^{3}$, $10^4$ and $5 \cdot 10^4$.  The graphs show both the nonpreconditioned (solid lines) as well as the preconditioned (dashed lines) version of each of the algorithms.  As the Figures \ref{fig:prec_AllenCahn} and \ref{fig:precSchnackenberg} demonstrate, method { \tt HImExp2J} is consistently the best performing scheme for all of the simulations for both problems.  For Allen-Cahn equation methods { \tt HImExp2N} and {\tt ImExpRK2}  are slower than {\tt 2-sBDF} for the parameters chosen here due to the higher computational cost per time step, but the largest time step that can be taken with these IMEXP schemes is an order of magnitude larger than the {\tt 2-sBDF} method (Table \ref{table:timestepSize}). It is feasible to imagine that there exist problems for which this will be advantageous.  In fact for the most stiff version of the 2D Schnakenberg system  (Figure \ref{fig:precSchnackenberg} (c,d)), { \tt HImExp2N} and {\tt ImExpRK2} begin to outperform {\tt 2-sBDF}.  {\tt HImExp2J}, however, remains the best performing scheme for the Schnakenberg's equation as well. As discussed in section 2.2 this is anticipated due to the fact that {\tt HImExp2J} can offer better stability when integrating stiff nonlinearities.

To detail the computational savings offered by the { \tt HImExp2J}  method compared to {\tt 2-sBDF}  we present Table \ref{table:CPUtimes} where for several given 
tolerances we compute the CPU time required by the { \tt HImExp2J}  scheme as a percentage of the CPU execution time taken by {\tt 2-sBDF}. When exact CPU time is not available for a given tolerance we interpolate its value from the corresponding precision graphs.  The data in the table clearly shows that the computational savings offered by { \tt HImExp2J} grow as the stiffness of the nonlinear operator $N$ is increased.  The savings as even more pronounced for the preconditioned versions of the algorithms since the computational cost per time step is decreased and consequently the ability to take a larger time step becomes more important.  Note that without the loss of generality in our conclusions we can vary the length of the integration interval. Since the total computational time to achieve accuracy of $10^{-3}$ with {\tt 2-sBDF} for the long time interval for the Schnakenberg system became too long (see Figure \ref{fig:precSchnackenberg}c), we decreased the total integration interval to be able to get CPU times ratio for the stiffest version of the problem with $\gamma = 5 \cdot 10^4$. 

The results of our numerical simulations clearly verifies the theoretical predictions of the performance of the new IMEXP schemes and presents these methods as a promising alternative to the IMEX integrators for problems with stiff nonlinearity $N$. We stress that ideas presented in the paper can be easily extended to construct many additional methods. For example, one could use similar techniques to address problems where operator $L(u)$ is nonlinear but a good preconditioner is available.  Other combinations of the uses of operators $L'$, $N'$ and $J$ are possible. We defer these developments to our future publications along with the construction of higher order IMEXP methods.

\begin{table}[H]
\begin{minipage}{\textwidth}
\centering
\caption{Approximated CPU time for {\tt HimExp2J}  as a percentage of the CPU time for {\tt 2-sBDF}  given a prescribed accuracy. }
\label{table:CPUtimes}
\begin{tabular}{c}
(a) Allen-Cahn $N=150^2$\\
{\begin{tabular}{clccc}
\hline
Accuracy                   & \multicolumn{1}{c}{} & $\epsilon=0.02$ & $\epsilon=0.01$ & $\epsilon=0.005$ \\ \hline
\multirow{2}{*}{$10^{-2}$} & Non-Preconditioned   & 55\%           & 56\%           &                44\% \\
                           & Preconditioned       & 51\%           & 56\%           &                44\% \\ \hline
\multirow{2}{*}{$10^{-3}$} & Non-Preconditioned   &    43\%            &     54\%           &                45\% \\
                           & Preconditioned       &     35\%          &    54\%           &                42\%\\ \hline
\end{tabular}
}\\

\\
(b) Schnakenberg $N=128^2$\\
{\begin{tabular}{clccc}
\hline
Accuracy                   & \multicolumn{1}{c}{} & $\gamma=1000$ & $\gamma=10000$ & $\gamma=50000$\footnote{$t_{\textrm{end}}=0.01$} \\ \hline
\multirow{2}{*}{$10^{-1}$} & Non-Preconditioned   &  55 \% &      33\%       &    33\% \\
                           & Preconditioned       &    171\%  &     31\%        &      24\% \\ \hline
\multirow{2}{*}{$10^{-2}$} & Non-Preconditioned   &       44\% &    32\%            &                27\% \\
                           & Preconditioned       &         120\%  &  22\%            &                20\% \\ \hline
  \multirow{2}{*}{$10^{-3}$} & Non-Preconditioned   &       48\%        &     19\%           &                24\% \\
                           & Preconditioned       &         70\%  &      13\%        &                20\% \\ \hline
\end{tabular}
}
\end{tabular}
\end{minipage}
\end{table}


\begin{figure}[H]
\centering
   \caption{2D Schnakenberg problem ($N=128^2$): CPU execution time versus error for constant time step}
    \label{fig:precSchnackenberg}
\begin{tabular}{C{.5\textwidth}R{2pt}C{.5\textwidth}}
(a)  $\gamma=10^3$, $t_{\textrm{end}}=0.1$ \newline [1.2ex]
{\includegraphics[width=2.8in,keepaspectratio]{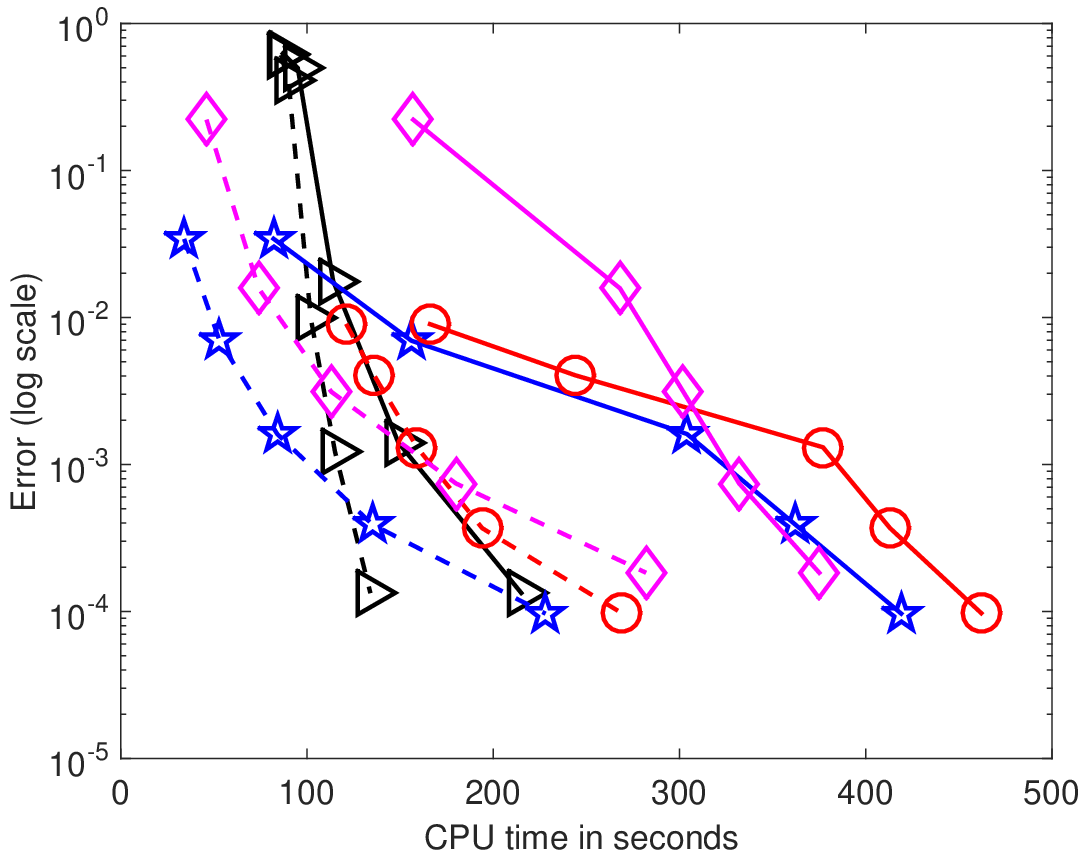} }& &  
(b) $\gamma=10^4$,  $t_{\textrm{end}}=0.1$ \newline [1.3ex]
{\includegraphics[width=2.85in,keepaspectratio]{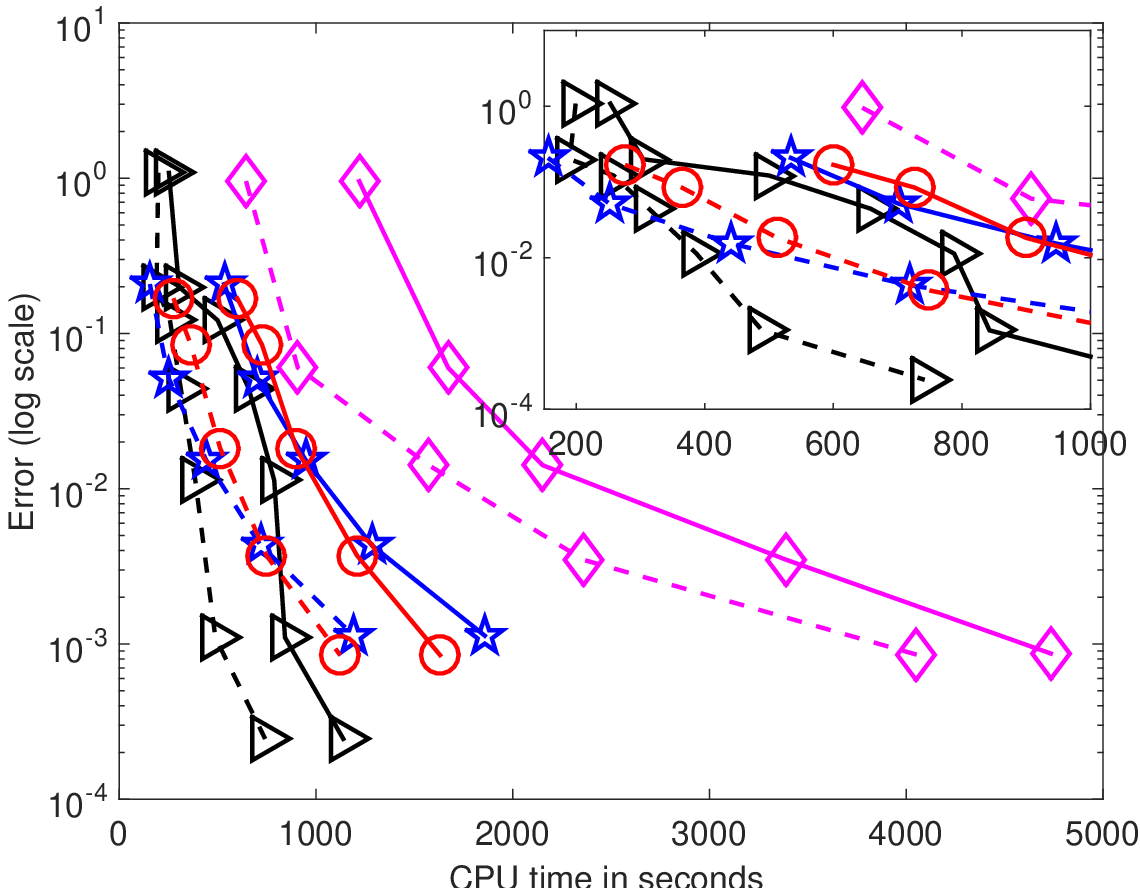} }\\[-1ex]
(c)   $\gamma =5\cdot 10^{4}$,  $t_{\textrm{end}}=0.1$  \newline [1.3ex]
{\includegraphics[width=2.8in,keepaspectratio]{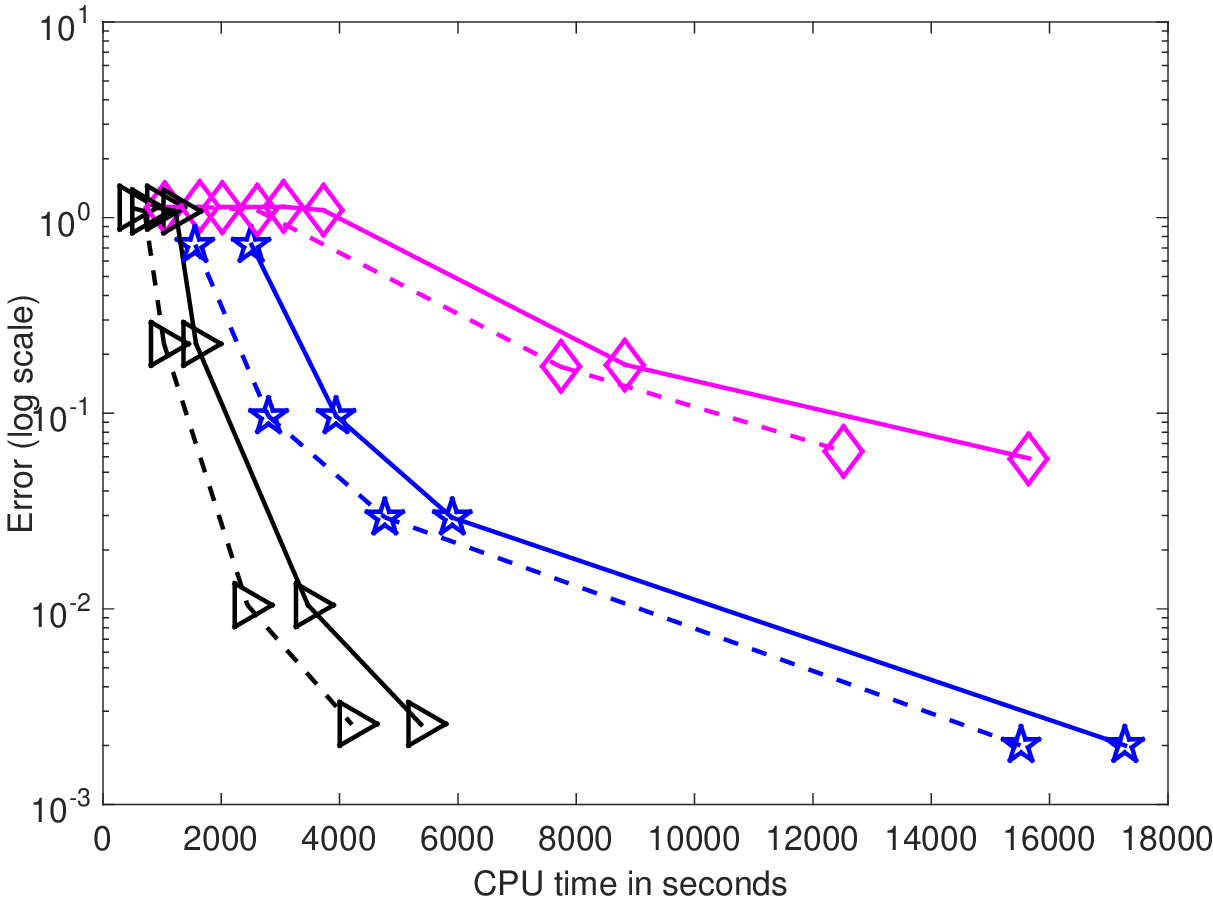} }\vspace{-2pt} & &   
{\hfill \includegraphics[scale=0.18]{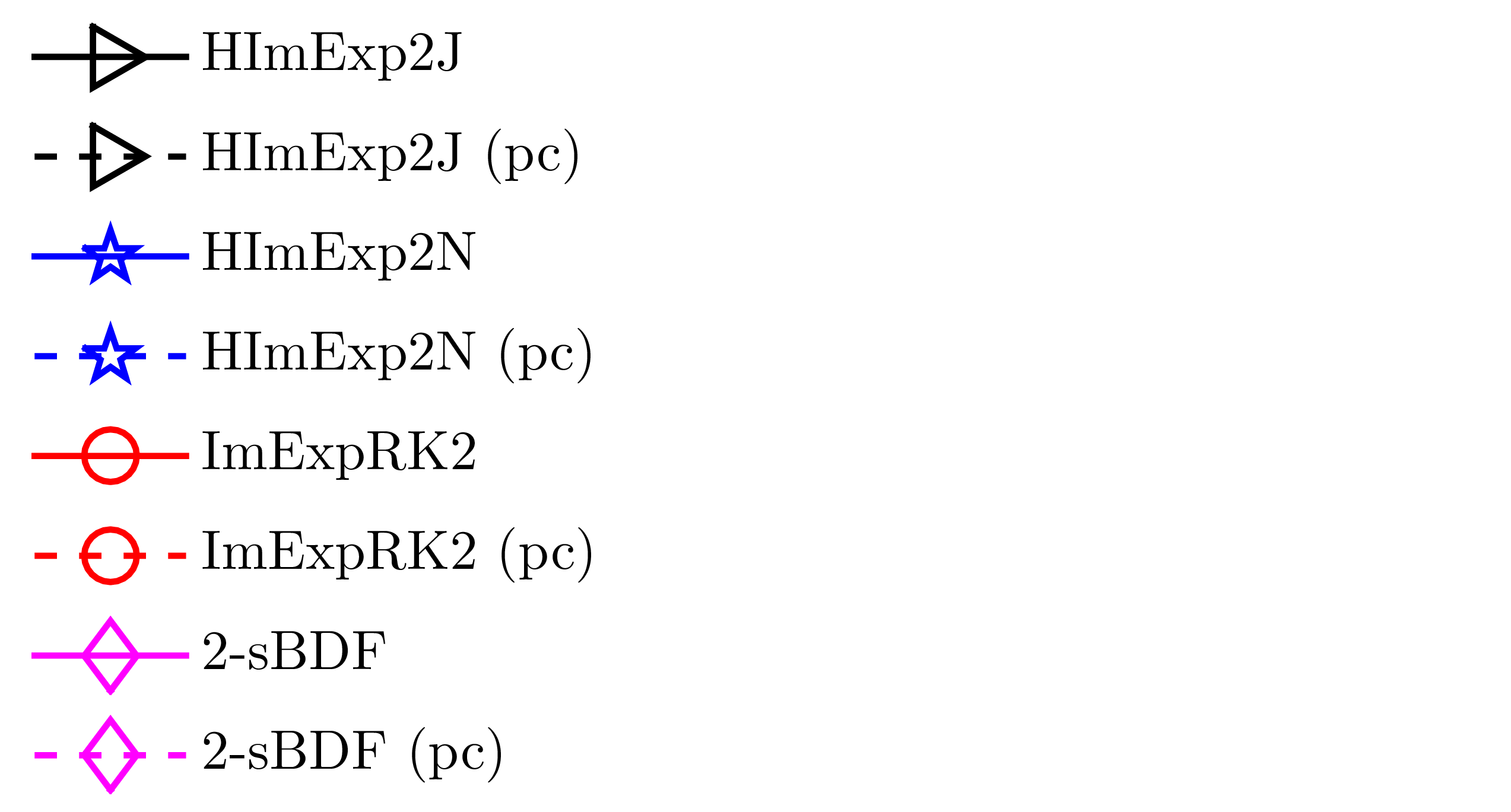} } \vspace{2pt}\\
(d)  $\gamma=5\cdot 10^{4}$, $t_{\textrm{end}}=0.01$  \newline [1.2ex]
{\includegraphics[width=2.8in,keepaspectratio]{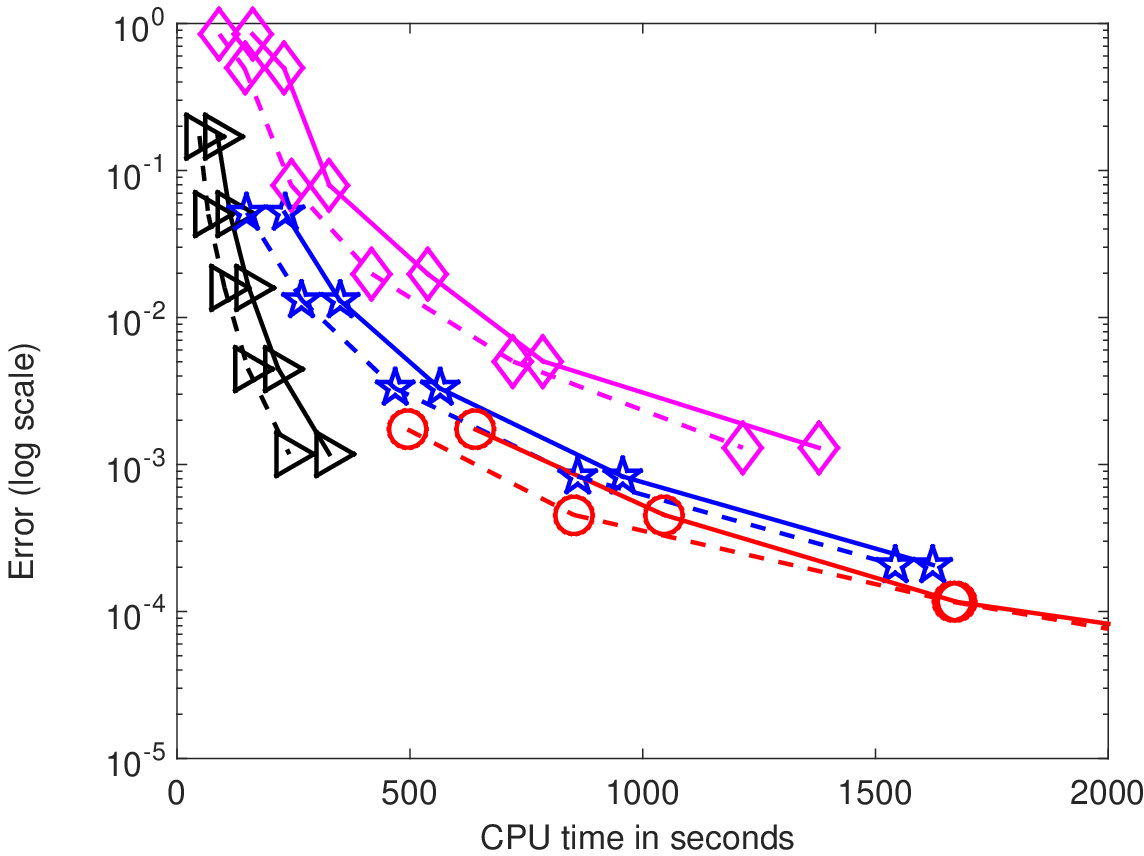} }
\end{tabular}
\end{figure}


\begin{figure}[H]
\centering
   \caption{2D Allen-Cahn problem ($N=150^2$):  CPU execution time versus error for constant time step}
    \label{fig:prec_AllenCahn}
\begin{tabular}{C{.5\textwidth}R{3pt}C{.5\textwidth}}
 (a) $\epsilon=0.02$ \newline [1.2ex]
   {\includegraphics[width=3in,keepaspectratio]{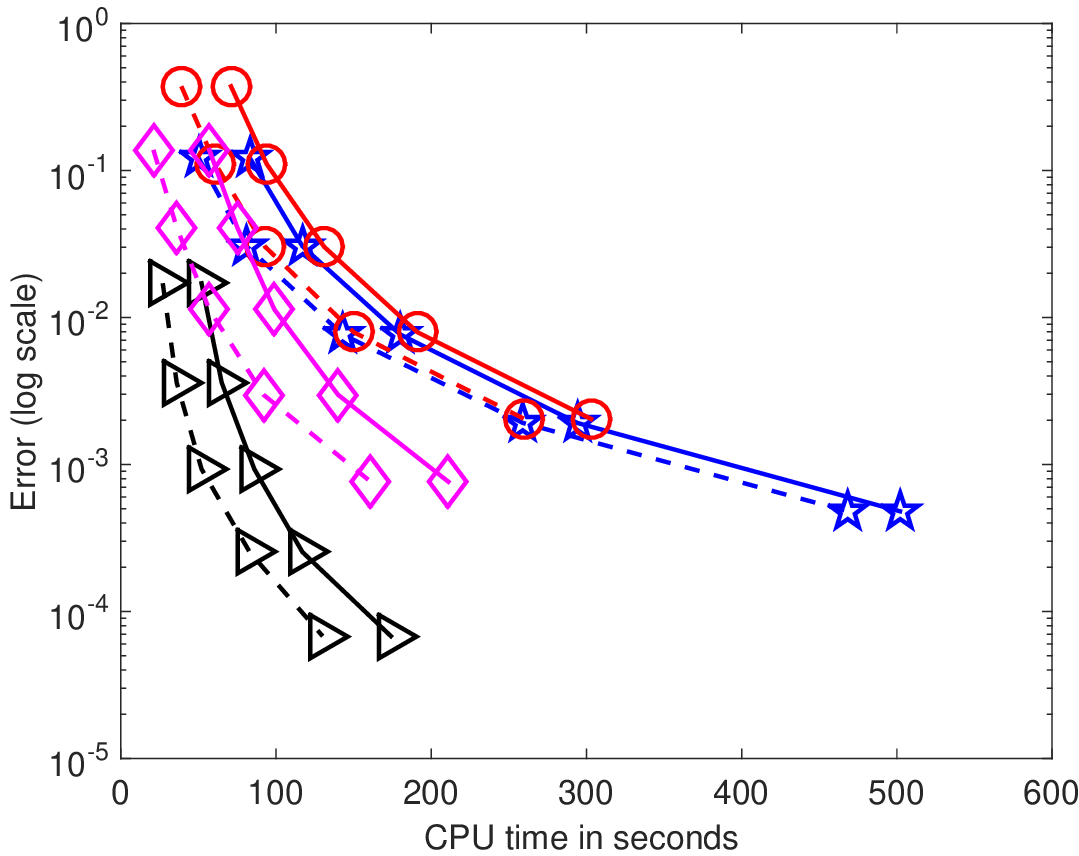} } &    & 
(b) $\epsilon=0.01$ \newline [1.3ex]
{\includegraphics[width=3in,keepaspectratio]{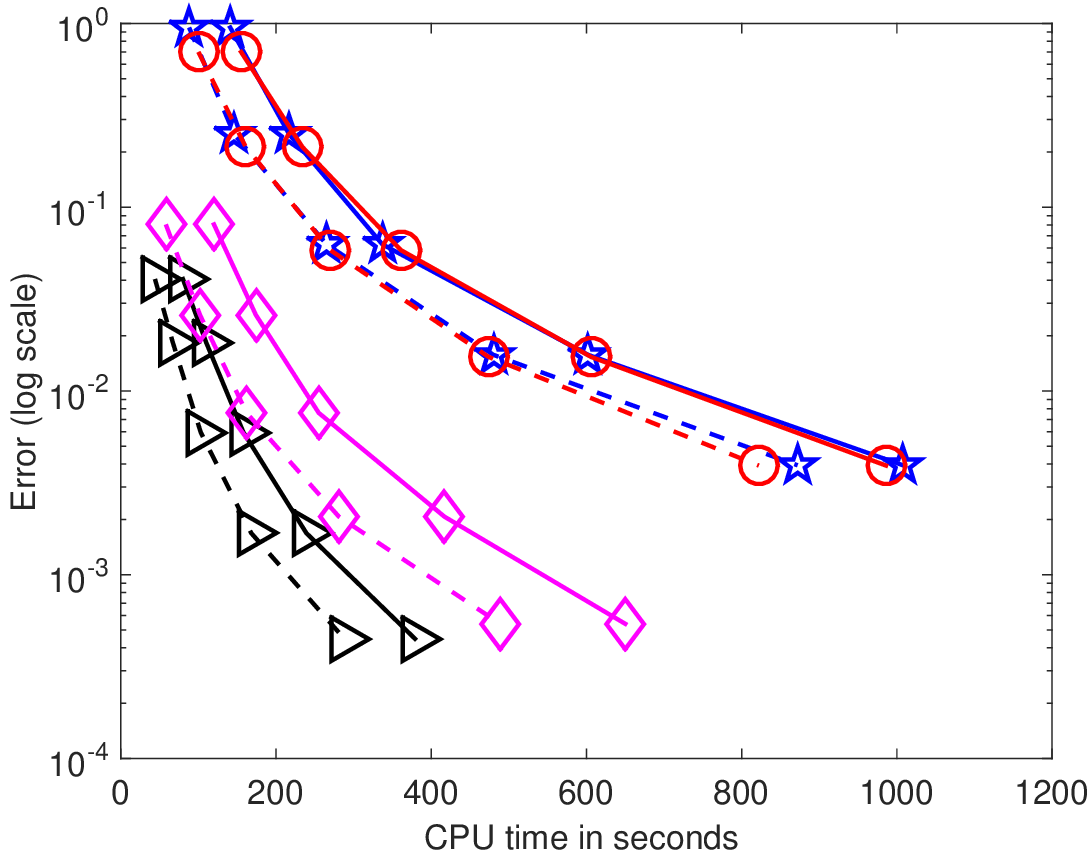}} \\
(c) $\epsilon=0.005$\newline [1.3ex]
{\includegraphics[width=3in,keepaspectratio]{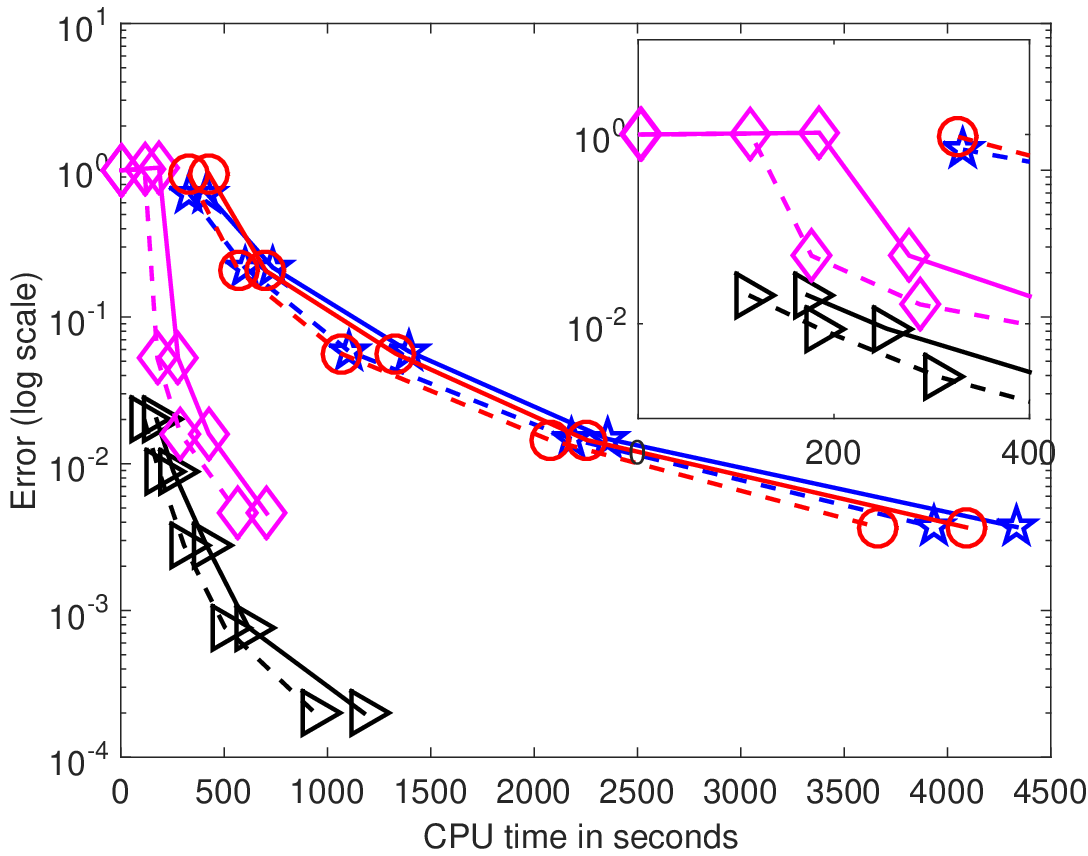} }&& 
{\hfill \includegraphics[scale=0.2]{Legend} } \vspace{2pt} \\
\end{tabular}
\end{figure}

\section*{Acknowledgements}
This work was supported by a grant from the National Science Foundation,
Computational Mathematics Program, under Grant No. 1115978.


\begin{thebibliography}{9}
\bibliographystyle{amsplain}


\bibitem{canuto87} 
C. Canuto, M. Hussaini, A. Quarteroni, T. Zang, Spectral methods in
fluid dynamics, Springer-Verlag, New York, 1987.

\bibitem{kim85}
J. Kim, P. Moin, Application of a fractional-step method to incompressible
Navier-Stokes equations, J. Comput. Phys. 59 (1985) 308--323.

 
  \bibitem{AscherRuuth}
 U. Ascher, S. Ruuth, B. Wetton, Implicit-explicit methods for timedependent
PDEs, SIAM J. Numer. Anal. 32 (3) (1997) 797--823.

  \bibitem{frank1997}
  J. V. J. Frank, W. Hundsorfer, On the stability of implicit-explicit linear
multistep methods, Appl. Numer. Math. 25 (1997) 193--205.

\bibitem{ascher1997}
U. M. Ascher, S. J. Ruuth, R. Spiteri, Implicit-explicit Runge–Kutta
methods for time-dependent partial differential equations, Appl. Numer.
Math. 25 (1997) 151--167.
 
 \bibitem{constantinescu2010}
E. Constantinescu, A. Sandu, Extrapolated implicit-explicit time stepping,
SIAM J. Sci. Comput. 31 (6) (2010) 4452--4477.


\bibitem{wathen15}
A. Wathen, Preconditioning, Acta Numer. 24 (2015) 329--376.


\bibitem{SBL12}
I.Sgura, B. Bozzini, D~.Lacitignola, 
Numerical approximation of Turing patterns in electrodeposition by ADI methods, 
J. Comput. Appl. Math. 236, 4132--4147 (2012).

\bibitem{bisetti}
F. Bisetti, Integration of large chemical kinetic mechanisms via exponential
methods with krylov approximations to jacobian matrix functions,
Combust. Theor. Model. 16 (3) (2012) 387--418.

\bibitem{tokman2001}
M. Tokman, Magnetohydrodynamic modeling of solar magnetic arcades
using exponential propagation methods, Ph.D. thesis, California Institute
of Technology (2001).

\bibitem{goedbloed2004}
H. Goedbloed, S. Poedts, Principles of Magnetohydrodynamics with Applications
to Laboratory and Astrophysical Plasmas., Cambridge University
Press, 2004.


\bibitem{hochbruckexp4}
M. Hochbruck, C. Lubich, H. Selhofer, Exponential integrators for large
systems of differential equations, SIAM J. Sci. Comput. 19 (1998) 1552--
1574.


\bibitem{Tok06}
M.~Tokman,  Efficient integration of large stiff systems of {ODE}s with exponential propagation iterative ({EPI}) methods, J. Comput. Phys., 213 (2006), pp.~748--776.


\bibitem{HOS09}
M.~Hochbruck, A.~Ostermann and J.~Schweitzer, Exponential Rosenbrock-type methods, SIAM J. Numer. Anal., 47 (2009), 786--803.

\bibitem{HO10}
M.~Hochbruck and A.~Ostermann,  Exponential integerators, Acta Numerica, 19 (2010), pp.~209--286.

\bibitem{TLP12}
M. Tokman, J. Loffeld, P. Tranquilli,
New adaptive exponential propagation iterative methods of Runge-Kutta type, SIAM J. Sci. Comput. 34(5), A2650--A2669 (2012).

\bibitem{LO14a}
V.T.~Luan and A.~Ostermann,
Exponential Rosenbrock methods of order five--construction, analysis and numerical comparisons,
J. Comput. Appl. Math., 255 (2014), 417--431. 


\bibitem{LO16}
V. T. Luan, A. Ostermann, Parallel exponential Rosenbrock methods,
Comput. Math. Appl. 71 (2016) 1137--1150.

\bibitem{loffeld14}
J. Loffeld, M. Tokman, Implementation of parallel adaptive-Krylov exponential
solvers for stiff problems, SIAM J. Sci. Comput. 36 (5) (2014)
C591--C616.

\bibitem{tokmanOberwolfach}
M. Tokman, Four classes of exponential epirk integrators, Oberwolfach
Reports (14) (2014) 855--858.

\bibitem{Tok11}
M.~Tokman, A new class of exponential propagation iterative methods of Runge--Kutta type (EPIRK), J. Comput. Phys., 230 (2011), pp.~8762--8778.

\bibitem{rainwater2016}
G. Rainwater, M. Tokman, A new approach to constructing efficient
stiffly accurate exponential propagation iterative methods of Runge-
Kutta type (EPIRK), ArXiv e-printsarXiv:1604.00583.

\bibitem{rainwater14}
G. Rainwater and M. Tokman, A new class of split exponential propagation iterative methods of {R}unge--{K}utta type (s{EPIRK}) for semilinear systems of ODEs,
J. Comput. Phys. 269, (2014), 40--60.

\bibitem{HO05}
M.~Hochbruck and A.~Ostermann,  Explicit exponential Runge--Kutta  methods for semilinear parabolic problems,  SIAM J. Numer. Anal., 43 (2005), 1069--1090.

\bibitem{LO12b}
V.T. Luan, A. Ostermann,
Stiff order conditions for exponential Runge--Kutta methods of order five,
in book: Modeling, Simulation and Optimization of Complex Processes - HPSC 2012 (H.G. Bock et al. eds.), 133--143 (2014).

\bibitem{LO14b}
V.T.~Luan and A.~Ostermann,
Explicit exponential Runge--Kutta methods of high order for parabolic problems,  J. Comput. Appl. Math., 256 (2014), 168--179.

\bibitem{EN2000}
K.-J.~Engel and R.~Nagel,  One-Parameter Semigroups for Linear Evolution Equations,  Springer, New York, 2000.


\bibitem{PAZY83}
A.~Pazy,  Semigroups of Linear Operators and Applications to Partial Differential Equations, Springer, New York, 1983.


\bibitem{H81}
D. Henry, Geometric theory of semilinear parabolic equations, Vol. 840
of Lecture Notes in Mathematics, Springer-Verlag Berlin Heidelberg,
1981.

\bibitem{LO13}
V.T.~Luan and A.~Ostermann,
Exponential B-series: The stiff case, SIAM J. Numer. Anal., 51, 3431--3445 (2013).



\bibitem{Em05}
E.~Emmrich, Stability and error of the variable two-step BDF for semilinear parabolic problems, J. Appl. Math. Comput. 19 (2005), pp.~33--55.


\bibitem{HW96}
E.~Hairer and G.~Wanner, Solving Ordinary Differential Equations II, Stiff and Differential-Algebraic Problems, 2nd rev. ed., Springer, New York, 1996.


\bibitem{Zhu2015}
L. Zhu, L. Ju, W. Zhao, Fast high-order compact exponential time differencing
Runge–Kutta methods for second-order semilinear parabolic
equations, J. Sci. Comput. 67 (3) (2016) 1043--1065.


 
 \bibitem{Ruuth}
S.~Ruuth, Implicit-explicit methods for reaction-diffusion problems in
 pattern formation, J. Mathematical Biology \textbf{34} (1995), 148--176.

\bibitem{schnakenberg}
J.~Schnakenberg, Simple chemical reaction systems with limit cycle
 behaviour, J. Theoretical Biology \textbf{81} (1979), 389--400.
 




 \bibitem{madzvamuse2007}
 A. Madzvamuse, P. Main, Velocity-induced numerical solutions of
reaction-diffusion systems on continuously growing domains, J. Comput.
Phys. 225 (2007) 100--119.












\end{thebibliography}

%
\end{document}